\documentclass[12pt,review]{article}
 \expandafter\let\csname equation*\endcsname\relax
 \expandafter\let\csname endequation*\endcsname\relax
\usepackage{bm}
\usepackage{enumerate}
\usepackage[top=1in, bottom=1in, left=1in, right=1in]{geometry}
\usepackage{mathrsfs,color}
\usepackage{amsfonts}\usepackage{multirow}
\usepackage{amssymb}
\usepackage{dsfont}
\usepackage{caption,comment}
\usepackage{blkarray}
\usepackage{amsmath}
\usepackage{float}
\usepackage{amssymb,amsthm}
\usepackage[toc,page]{appendix}
\usepackage{graphicx,afterpage}
\usepackage{epstopdf}
\usepackage{wrapfig}
\usepackage{cancel}
\usepackage{mathdots}
\usepackage[colorlinks=true]{hyperref}

\usepackage{todonotes}
\hypersetup{CJKbookmarks,%
bookmarksnumbered,%
colorlinks,%
linkcolor=black,%
citecolor=black,%
plainpages=false,%
pdfstartview=FitH}
\numberwithin{equation}{section}
\numberwithin{figure}{section}

\newcommand\tabcaption{\def\@captype{table}\caption}

\newtheorem{thm}{Theorem}[section]
\newtheorem{cor}[thm]{Corollary}
\newtheorem{lem}[thm]{Lemma}
\newtheorem{prop}[thm]{Proposition}
\newtheorem{defn}[thm]{Definition}

\newtheorem{rem}[thm]{Remark}

\newcommand{\calI}{\mathcal{I}}

\newcommand{\bfe}{\mathbf{e}}
\newcommand{\Dcut}{\mathbf{D}_{cut}}

\newcommand{\Xhat}{\widehat{X}}

\newcommand{\Chat}{\widehat{C}}

\newcommand{\E}{\mathbb{E}}

\newcommand{\Prob}{\mathbb{P}}

\newcommand{\reals}{\mathbb{R}}

\newcommand{\unit}{\mathds{1}}

\newcommand{\Xbar}{\overline{X}}

\newcommand{\Rhat}{\widehat{R}}

\newcommand{\dist}{\mathbf{d}}

\newcommand{\Khat}{\widehat{K}}

\newcommand{\ehat}{\hat{e}}
\newcommand{\bfP}{\mathbf{P}}
\newcommand{\bfD}{\mathbf{D}}

\newcommand{\Riccati}{\mathcal{R}}

\newcommand{\Mhat}{\widehat{M}}

\newcommand{\calB}{\mathcal{B}}

\begin{document}
\title{Performance analysis of local ensemble Kalman filter}
\author{Xin T. Tong \thanks{National University of Singapore, mattxin@nus.edu.sg}}
\maketitle
\begin{abstract}
Ensemble Kalman filter (EnKF) is an important data assimilation method for high dimensional geophysical systems.
Efficient implementation of EnKF in practice often involves the localization technique, 
which updates each component using only information within a local radius. 
This paper rigorously analyzes the local EnKF (LEnKF) for linear systems, 
and shows that the filter error can be dominated by the ensemble covariance,
as long as 1) the sample size exceeds the logarithmic of state dimension and a constant that depends only on the local radius;
2) the forecast covariance matrix admits a stable localized structure.
In particular, this indicates that with small system and observation noises, the filter error will be accurate in long time even if the initialization is not.
The analysis also reveals an intrinsic inconsistency caused by the localization technique, 
and a stable localized structure  is necessary to control this inconsistency.
While this structure is usually taken for granted for the operation of LEnKF,
it can also be rigorously proved for linear systems with sparse local observations and weak local interactions. 
These theoretical results are also validated by numerical implementation of LEnKF on a simple stochastic turbulence in two dynamical regimes. 
\end{abstract}

\section{Introduction}
\label{sec:intro}
Data assimilation is a sequential procedure, in which observations of a dynamical system are incorporated to improve the forecasts of that system. 
In many of its most important geoscience and engineering applications, the main challenge  comes from the high dimensionality of the system. 
For contemporary atmospheric models, the dimension can reach $d\sim 10^8$, and the classical particle filter is no longer feasible \cite{SBB08, vLee09}. 
The ensemble Kalman filter (EnKF) was invented by meteorologists \cite{And01, HWS01, evensen03} to resolve this issue. 
By sampling the forecast uncertainty with a small ensemble, and then employing  Kalman filter procedures to the empirical distribution, EnKF can often capture the major uncertainty and produce accurate predictions. The simplicity and efficiency of EnKF have made it a popular choice for weather forecasting and oil reservoir management \cite{MH12, kalnay03}.

One fundamental technique employed by EnKF is localization \cite{HM98, HWS01, WH02, MY07, HKS07}. 
In most geophysical applications, each component $[X]_i$ of the state variable $X$ holds information of one spatial location. 
There is  a natural distance $\dist(i,j)$ between two components. 
In most physical systems, the covariance between  $[X]_i$ and $[X]_j$ is formed by information propagation in space, 
intuitively its strength decays with the distance $\dist(i,j)$. 
In particular, when $\dist(i,j)$ exceeds a threshold $L$, 
the covariance is approximately zero.
This is a special sparse and localized structure
that can be exploited in the EnKF operation.
In particular, the forecast covariance can be artificially enforced as zero if $\dist(i,j)>L$.
In other words, there is no need to sample these covariance terms, and indeed sampling from them leads to higher errors  \cite{HWS01}.
Such modification significantly reduces the sampling difficulty and the associated sample size.
 This is crucial for EnKF operation, since often only a few hundred samples can be generated in practice.
Various versions of localized EnKF (LEnKF) are derived based on this principle, 
and there is ample numerical evidence showing their performance is robust against the growth of dimension  \cite{HWS01, WH02, MY07, HKS07, BR10, JNASS11, NJSH12, Ner15, KR16, CH17}. 
Moreover, there is a growing interest in applying the same technique to the classical particle filters \cite{VR15, ML16, Pot16}. 

While there is a consensus on the importance of the localization technique for EnKF, 
currently there is no rigorous explanation of its success. 
This paper contributes to this issue by showing that in the long run, the LEnKF can reach its estimated performance for linear systems,
if the ensemble size $K$ exceeds $D_L \log d$, and the ensemble covariance matrix admits a stable localized structure of radius $L$. 
The  constant $D_L$  above depends on the radius $L$ but not on $d$. 

Showing the necessary sampling size has only logarithmic dependence on $d$ is our major interest. 
In the simpler scenario of sampling a static covariance matrix,
 \cite{BL08} shows that the necessary sample size scales with $D_L \log d$.
Generalizing this result to the setting of EnKF is highly nontrivial, since
the target covariance matrix evolves constantly in time, 
and the sampling error at one time step has a  nonlinear impact on future iterations.
By analyzing the filter forecast error evolution, 
and compare it with the filter covariance evolution, 
we show the filter error covariance can be dominated by the ensemble covariance with high probability.
In other words, the LEnKF can reach its estimated performance.
One important corollary  is that if the system and observation noise are of scale $\sqrt{\epsilon}$, 
then the error covariance scales as $\epsilon$, which indicates that LEnKF can be accurate regardless of the initial condition. 
Such property is often termed as accuracy for practical filters or observers \cite{SS15, MT17cpam, KS17}.  


Interestingly, our analysis also captures an intrinsic inconsistency caused by the localization technique.
Generally speaking, the localization technique can be applied to the ensemble covariance matrix,
but not the ensemble. 
However, the Kalman update is applied to the ensemble, but not to the localized ensemble covariance matrix. 
As these two operations do not commute, an inconsistency emerges, which we will call the localization inconsistency.
This phenomenon has been mentioned in \cite{WH02, WHWST08}.
Moreover,  \cite{Ner15} numerically examines its role with serial observation processing, 
and shows that it may lead to significant filter error.
In correspondence to these findings, one crucial step in our analysis is showing that the localization inconsistency is controllable, 
if the forecast covariance matrix indeed has a localized structure.

While most applications of LEnKF assume the underlying covariance matrices are localized, 
rigorous justification of this assumption is sorely missing in the literature. 
A recent work \cite{BMP17} considers applying a projection to the continuous time Kalman-Bucy filter,  
and shows that if the projection is a small perturbation on the covariance matrix, its impact on the filter process is also small.
It is shown through an example that if the filter system can be decoupled into independent local parts,
a projection similar to the LEnKF localization procedure can be made. 
Unfortunately, in most practical problems, all spatial dimensions are coupled with local interactions,
and it is very difficult to show that the localization procedure is a small perturbation. 

This paper partially investigates the theoretical gaps mentioned above. 
We show that for linear systems with weak local interactions and sparse local observations, 
the localized structure is stable for the LEnKF ensemble covariance. 
Weak local interaction is an intuitive requirement, 
else fast information  propagation will form strong covariances between far away locations.
Sparse local observation, on the other hand, is assumed to simplify the assimilation formulas. 

In rough words, our main results consist of the following statements.
\begin{enumerate}
\item To sample a localized covariance matrix correctly, the necessary sample size scales with $D_L\log d$ (Theorem \ref{thm:localconcen}).
This reveals the sampling advantage gained by applying the localization procedure. 
\item While localization improves the sampling, it creates an inconsistency in the assimilation steps.  For the LEnKF ensemble covariance to capture the  filter error covariance with $D_L\log d$ samples, the localization inconsistency needs to be small  (Theorem \ref{thm:main}).
\item One way to guarantee a small localization inconsistency, is to have a stable localized structure in the forecast ensemble covariance matrix (Proposition \ref{lem:Kalmanlocalizationerror}).
\item The LEnKF  forecast covariance has a stable localized structure, if the underlying linear system has weak  interactions and sparse local observations.   (Theorem \ref{thm:formloc}). So by points 2 and 3, we know that LEnKF has good forecast skills, since its ensemble covariance captures the true filter error covariance. 
\item The results above scale linearly with the variance of the  noises. So when applying LEnKF to a linear system with small system and observation noises, its long time performance is accurate  (Theorem \ref{thm:accuracy}). 
\end{enumerate}
Section \ref{sec:mainresult} will provide the setup of our problem, and present the precise statements of the main results. The implication of these results on the issue of localized radius is discussed in Section \ref{sec:radius}. 

Section \ref{sec:num} verifies the theoretical results by implementing LEnKF on a stochastically forced dissipative advection equation \cite{MH12}. One stable and one unstable dynamical regimes are tested. In both of them, LEnKF have shown robust forecast skill with only $K=10$ ensemble members, while the dimension varies between $10$ and $1000$. Moreover the localized covariance structure and the accuracy with small noises can also be verified for LEnKF in both regimes.

Section \ref{sec:concen} investigates the covariance sampling problem of LEnKF, and proves Theorem \ref{thm:localconcen}. Section \ref{sec:error} analyzes the localization inconsistency and filter error evolution. It contains the proofs of Theorem \ref{thm:main} and 
Proposition \ref{lem:Kalmanlocalizationerror}. Section \ref{sec:closeneighbor} studies the localized structure of linear systems with weak local interactions and sparse observations, and  shows that the small noise scaling can be applied to our results. Section \ref{sec:conclude} concludes this paper and discusses some interesting extensions.

\section{Main Results}
\label{sec:mainresult}
\subsection{Problem Setup}
\label{sec:setup}
Since its invention, the ensemble Kalman filter (EnKF) has been modified constantly for two decades, and its formulation has become rather sophisticated today. In this subsection we briefly review some of the key modifications, in particular the localization techniques. 

The following notations will be used throughout the paper. For two vectors $a$ and $b$, $\|a\|$ denotes the $l_2$ norm of $a$, $a\otimes b$ denotes the matrix $a b^T$. Square bracket with subscripts indicates a component or entry of an object. So $[a]_i$ is the $i$-th component of vector $a$. In particular, we use $\bfe_i$ to denote the $i$-th standard basis vector, i.e. $[\bfe_i]_j=\unit_{i=j}$. 

Given a matrix $A$, $[A]_{i,j}$ is the $(i,j)$-th entry of $A$. The   $l_2$ operator norm is denoted by $\|A\|=\inf\{c: \|Av\|\leq c\|v\|,\forall v\}$. The $l_\infty$ operator norm is denoted by $\|A\|_1=\max_i \sum_{j} |[A]_{i,j}|$. The maximum absolute entry is denoted by $\|A\|_\infty=\max_{i,j}|[A]_{i,j}|$.  We also use $I_m$ to denote the $m\times m$ dimensional identity matrix.  Given two matrices $A$ and $D$, their Schur (Hadamard) product can be defined by entry wise product
\[
[A\circ D]_{i,j}=[A]_{i,j}[D]_{i,j}.
\]
For two real symmetric matrices $A$ and $B$, $A\preceq B$ indicates that $B-A$ is positive semidefinite.

\subsubsection*{Ensemble Kalman Filter}
\label{sec:EnKF}
In this paper, we consider a linear system in $\reals^d$ with partial observations,
\begin{equation}
\label{sys:signalobs}
\begin{gathered}
 X_{n+1}=A_n X_n+b_n +\xi_n,\quad \xi_{n+1}\sim \mathcal{N}(0, \Sigma_n),\\
 Y_{n+1}=H X_{n+1}+\zeta_n,\quad \zeta_{n+1}\sim \mathcal{N}(0, \sigma_o^2 I_q). 
\end{gathered}
\end{equation}
Throughout our discussion, we assume the matrices $A_n,\Sigma_n$ are bounded:  
\[
\|A_n\|\leq M_A,\quad m_\Sigma I_d\preceq \Sigma_n\preceq M_\Sigma I_d.
\]
The time-inhomogeneous generality can be used to model intermittent dynamical systems \cite{MH12, MT16uq}. We assume that the observations are made at $q<d$ distinct locations $\{o_1, o_2,\cdots, o_q\}\subset \{1,\cdots,d\}$. This can be modelled by letting
\begin{equation}
\label{sys:observation}
[H]_{k,j}=\unit_{j=o_k},\quad 1\leq k\leq q, 1\leq j\leq d. 
\end{equation}
Note  that the operator norm $\|H\|=1$. 

It is well known that the optimal estimate of $X_n$ given historical observations $Y_1,\ldots, Y_n$ is provided by the Kalman filter \cite{ LS01}, assuming $X_0$ is Gaussian distributed. Unfortunately, direct implementation of the Kalman filter involves a stepwise computation complexity of $O(d^2 q)$. When the state dimension $d$ is high, the Kalman filter is not computationally feasible. 

The ensemble Kalman filter (EnKF) is invented by meteorologists \cite{evensen03} to reduce the computation complexity. $K$ samples of \eqref{sys:signalobs} are updated using the Kalman filter rules, and their ensemble mean and covariance are employed to estimate the signal $X_n$. In specific, suppose the posterior ensemble for $X_n$ is denoted by $\{X_n^{(k)}\}_{k=1,\ldots, K}$. The forecast ensemble of $X_{n+1}$ is first generated by propagating the linear system in \eqref{sys:signalobs}:
\[
\Xhat^{(k)}_{n+1}=A_n X^{(k)}_n+b_n+\xi^{(k)}_{n+1},\quad \xi^{(k)}_{n+1}\sim \mathcal{N}(0, \Sigma_n).\\
\]
The EnKF then estimates $X_{n+1}$ with a prior distribution $\mathcal{N}(\overline{\Xhat}_{n+1}, \Chat_{n+1})$, where the mean and covariance are obtained by the forecast ensemble:
\[
\overline{\Xhat}_{n+1}=\frac1K\sum_{k=1}^K \Xhat^{(k)}_{n+1},\quad \Delta \Xhat^{(k)}_{n+1}:=\Xhat^{(k)}_{n+1}-\overline{\Xhat}_{n+1},\quad \Chat_{n+1}=\frac1K\sum_{k=1}^K \Delta \Xhat^{(k)}_{n+1}\otimes \Delta \Xhat^{(k)}_{n+1}. 
\]
Applying the Bayes' formula to the prior distribution and the linear observation $Y_{n+1}$, a target Gaussian posterior distribution for $X_{n+1}$ can be obtained. There are several ways to update the forecast ensemble so its statistics approximate the target ones. Here we consider the standard EnKF in \cite{evensen03, MH12} with artificial perturbations:
\begin{equation}
\label{eqn:assimil}
X^{(k)}_{n+1}=(I-\widetilde{K}_{n+1} H ) \Xhat^{(k)}_{n+1}  +\widetilde{K}_{n+1}  Y_{n+1} -\Khat_{n+1}\zeta^{(k)}_{n+1}.
\end{equation}
The Kalman gain matrix is given by $\widetilde{K}_{n+1}=\Chat_{n+1} H^T(\sigma_o^2 I_q+ H \Chat_{n+1} H^T)^{-1}$. The $\zeta^{(k)}_{n+1}$ are independent noises sampled from $\mathcal{N}(0, \sigma_o^2I_q)$.

The computation complexity of EnKF is roughly $O(K^2 d)$, assuming $A_n$ and $\Sigma_n$ are sparse \cite{Man06}. In practice,  the ensemble size $K$ is often less than a few hundred, so the operational speed is significantly improved. On the other hand, with the sample size $K$ much smaller than the state space dimension $d$, the sample covariance $\Chat_{n+1}$ often produces spurious correlations \cite{BLE98, evensen03}. Spurious correlations may seriously reduce the filter accuracy, since the Kalman filter operation hinges heavily on the correctness of covariance estimation. The localization techniques are often employed to resolve such problems. 

\subsubsection*{Localization techniques}
 In most geophysical applications, each dimension index $i\in \{1,\ldots, d\}$ corresponds to a  spatial location.  For simplicity, we assume different indices correspond to different spatial locations. Let $\dist(i,j)$ be the spatial distance between the locations $i$ and $j$ specify, then $\dist$ is also a distance on the index set $\{1,\ldots,d\}$. In other words,
 \begin{itemize}
 \item $\dist(i,j)=0$ if and only if $i=j$;
 \item  $\dist(i,j)=\dist(j,i)$;
 \item  $\dist(i,j)+\dist(j,k)\geq \dist(i,k)$.
 \end{itemize} 
 For a simple example,  one can correspond  index $i$ with  the integer $i$, then $\dist(i,j)=|i-j|$ clearly defines a distance. 

For most geophysical problems that can be  modeled by a (stochastic) partial differential equation, the covariance between two locations is caused by the propagation of information through local interactions. Information often is also dissipated during its propagation, so  its impact gets less significant when it reaches far-away locations. This leads to a localized covariance structure. In other words, there is a decreasing function $\phi:[0, \infty)\mapsto [0,1]$, $\phi(0)=1$ such that 
\[
[C_n]_{i,j}\propto \phi (\dist(i,j)).
\]
In geophysical applications, a localization radius $l$ is often defined, so $\phi(x)=0$ for $x>l$. Consequentially, it is natural to model the localization function as 
\begin{equation}
\label{eqn:local}
[\bfD_l]_{i,j}=\phi(\dist(i,j)).
\end{equation}
In particular,  the widely used  Gaspari-Cohn matrix \cite{GC99}  is of this form with
\begin{equation}
\label{eqn:GC}
\phi(x)=\left(1+\frac{x}{c_l}\right) \exp\left(-  \frac{x}{c_l}\right) \unit_{x\leq l},
\end{equation}
where the radius is often picked with $l=\sqrt{10/3}c_l$ or $2c_l$ \cite{Lor03}. Another simple localization matrix  corresponds to the cutoff or heavyside function $\phi(x)=\unit_{x\leq l}$, and we denote it by $\Dcut^l$. In other words
\begin{equation}
\label{eqn:Dcut}
[\Dcut^l]_{i,j}=\unit_{\dist(i,j)\leq l}.
\end{equation}
As a remark, while \eqref{eqn:GC} is more useful in practice, \eqref{eqn:Dcut} is much simpler for theoretical analysis and interpretation. Most of our analysis results in below only apply to \eqref{eqn:Dcut}, except Theorem \ref{thm:localconcen}. It will be very interesting to generalize the analysis framework here for localization functions like \eqref{eqn:GC}. 

The notion of  localization radius is closely related to the  \emph{bandwidth} of a matrix \cite{BL12}. For a matrix $A$, we define its bandwidth as:
\begin{equation}
\label{eqn:radius}
l:=\inf\{x\geq 0: [A]_{i,j}=0\quad \text{if}\quad \dist(i,j)>x\}.
\end{equation}
The bandwidth roughly captures how fast different components interact with each other. If $A$ has bandwidth $l$,  each component interacts with at most $\calB_l$ components when product with $A$, where the volume constant $\calB_l$ is defined by
\begin{equation}
\label{eqn:calB}
\calB_l=\max_i \#\{j: \dist(i,j)\leq l\}.
\end{equation}

A localized covariance  structure is extremely useful for EnKF. It indicates only covariances between nearby indices are worth sampling. By ignoring the far apart covariances,  the  necessary sampling size can be significantly reduced. To apply this idea, the localization technique modifies the Kalman gain matrix in \eqref{eqn:assimil}, and ensures the assimilation updates from far away observation is insignificant. There are  two main types of localization methods in the literature, domain localization and covariance localization \cite{NJSH12}. This paper discusses only the former, while similar analysis should in principal applies to the latter as well.

With domain localization, the $i$-th component is updated using only observations of indices within distance $l$, which are elements of $\calI_i=\{j: \dist(i,j)\leq l\}$. Let $\bfP_{\calI_i}$ be the projection matrix of  a $\reals^d$ vector to its components on $\calI_i$, note that it is diagonal so it is symmetric. Then $\Chat^i_{n+1}:=\bfP_{\calI_i} \Chat_{n+1} \bfP_{\calI_i}$  contains the local covariance relevant to the $i$-th component. The corresponding Kalman gain is 
\begin{equation}
\label{eqn:Ki}
K^i_{n+1}=\Chat^i_{n+1} H^T(\sigma_o^2I_q+H \Chat^i_{n+1}H^T )^{-1},
\end{equation}
and the $i$-th component is updated using the $i$-th row of \eqref{eqn:Ki}, namely $\bfe_i \bfe_i^T K^i_{n+1}$. Again $\bfe_i$  is the $i$-th standard basis vector of $\reals^d$. The final Kalman gain matrix patches all rows together
\begin{equation}
\label{eqn:Khat1}
\Khat_{n+1}=\sum_{i=1}^d  \bfe_i \bfe_i^T K^i_{n+1}.
\end{equation}
Since each $K^i_{n+1}$ has nonzero entries only with indies in $\calI_i\times \calI_i$, $\Khat_{n+1}H$ is of bandwidth $l$ as well. The proof in Proposition \ref{lem:Kalmanlocalizationerror} below verifies this. Therefore, each component is updated using observations of distance at most $l$ from it.

\subsubsection*{Localized EnKF with covariance inflation}
Other than spurious correlations, a small sampling size also jeopardizes the EnKF operation, as the forecast covariance is often undervalued \cite{FB07, LKM09, MT17cpam}. In order to resolve this issue, the covariance needs to be inflated with a fixed ratio $r>1$. \cite{MT17cpam} has shown these modification are pivotal to EnKF performance. We also incorporate this idea in our LEnKF. 

In summary, the localized EnKF (LEnKF) updates an posterior ensemble $\{X_n^{(k)},k=1,\cdots, K\}$ of its mean $\Xbar_n=\frac1K\sum_{k=1}^K X^{(k)}_n$  and spread $\Delta X_n^{(k)}=X_n^{(k)}-\Xbar_n$ through the following steps with $\Khat_{n+1}$ given by  \eqref{eqn:Ki} and \eqref{eqn:Khat1}:
\begin{equation}
\label{sys:EnKFloc}
\begin{gathered}
\overline{\Xhat}_{n+1}= A_n \Xbar_n+b_n,\quad \Delta \Xhat^{(k)}_{n+1}=\sqrt{r}(A_n \Delta X^{(k)}_n+\xi^{(k)}_{n+1}),\quad \xi^{(k)}_{n+1}\sim \mathcal{N}(0, \Sigma_n),\\
\Chat_{n+1}=\frac1K\sum_{k=1}^K \Delta \Xhat^{(k)}_{n+1}\otimes \Delta \Xhat^{(k)}_{n+1},\quad \Xbar_{n+1}=(I-\Khat_{n+1} H ) \overline{\Xhat}_{n+1}+\Khat_{n+1} Y_{n+1},\\
\Delta X^{(k)}_{n+1}=(I-\Khat_{n+1} H ) \Delta \Xhat^{(k)}_{n+1} +\Khat_{n+1}\zeta^{(k)}_{n+1}, \quad \zeta^{(k)}_{n+1}\sim\mathcal{N}(0, \sigma_o^2 I_q). 
\end{gathered}
\end{equation}
The posterior covariance matrix can be obtained through the spread
\[
C_{n+1}=\frac1K\sum_{k=1}^K \Delta X^{(k)}_{n+1}\otimes \Delta X^{(k)}_{n+1}.
\]
Note here  we  update the mean and ensemble spread, the $\Delta$ terms, separately. This is different from the standard EnKF, since the average noise terms $\frac1K\sum  \xi^{(k)}_{n+1}$ and $\frac1K\sum  \zeta^{(k)}_{n+1}$ are ignored for simplicity. Also the sum of the ensemble spread, $\sum \Delta X_{n}^{(k)}$, may not be zero. On the other hand, these differences are small by the law of large numbers. The proofs can also be generalized to admit these terms, but the discussion will be notationally complicated.  

One classical property of the Kalman filter is that the filter covariances and the Kalman gain matrices are predetermined with no dependence on the realization of system \eqref{sys:signalobs}. This is inherited by the LEnKF \eqref{sys:EnKFloc},  the covariances and Kalman gain depend only on the sample noise $\xi_n^{(k)}, \zeta_n^{(k)}$ realizations, but not on $(X_n, Y_n)$. 

To illustrate, consider the filtration generated by sample noise realization, 
\begin{equation}
\label{eqn:filterationS}
\mathcal{F}^S_{n}=\sigma\{\Delta \Xhat^{(k)}_0, \xi_{t}^{(k)},\zeta_{t-1}^{(k)}, t=1,\ldots, n, k=1,\ldots, K\}.
\end{equation}
Using induction, it is easy to verify the ensemble spread, ensemble covariance and Kalman gain, are all $\mathcal{F}^S_n$ adapted:
\[
\Delta \Xhat^{(k)}_{n}, \Delta X^{(k)}_{n-1}, \Chat_{n}, C_{n-1}, \Khat_{n}\in \mathcal{F}^S_n. 
\]
The corresponding conditional expectation is denoted by $\E_{\mathcal{F}^S_n}$. We will use $\mathcal{F}^S_\infty=\bigvee\mathcal{F}^S_n$ to denote the $\sigma$-field for all ensemble spread information.  

The other randomness of EnKF comes from the realization of system \eqref{sys:signalobs}. We can average out this part of randomness by conditioning on $\mathcal{F}^S_\infty$, which we will denote as $\E_S$. This is useful when comparing the filter error and sample covariance. The natural filtration generated by all random outcome at time $n$ is 
\[
\mathcal{F}_{n}=\sigma\{X_0, \Xhat^{(k)}_0,\xi_{t},\zeta_{t-1},  \xi_{t}^{(k)},\zeta_{t-1}^{(k)}, t=1,\ldots, n, k=1,\cdots,K\}.
\]
We will denote the conditional expectation with $\mathcal{F}_n$ as $\E_n$.

\subsection{Sampling errors of localized forecast covariance}
Since EnKF relies on the ensemble forecast covariance matrix to assimilate new observations, its performance depends on the accuracy of the sampling procedure. The sampling procedure updates the forecast matrix from time $n$ to $n+1$. 

Given the forecast ensemble covariance $\Chat_n$, based on the Kalman update rule, the inflated target  forecast covariance at $n+1$ is given by $r\Riccati_n(\Chat_n)$, with the posterior Riccati map
\begin{equation}
\label{eqn:Riccati}
\Riccati_n(\Chat_n):=A_n (I-\Khat_nH)\Chat_n (I-\Khat_nH)^TA_n^T+\sigma_o^2 A_n\Khat_n \Khat_n^TA^T_n+\Sigma_n. 
\end{equation}
The real ensemble  forecast covariance $\Chat_{n+1}=\frac1K \sum \Delta \Xhat^{(k)}_{n+1}\otimes \Delta \Xhat^{(k)}_{n+1}$  is  generated by the ensemble spread
\begin{equation}
\label{tmp:Delta}
\Delta \Xhat^{(k)}_{n+1}=\sqrt{r} A_n (I-\Khat_nH) \Delta \Xhat^{(k)}_n+\sqrt{r} A_n\Khat_n \zeta^{(k)}_n+ \sqrt{r} \xi^{(k)}_{n+1}. 
\end{equation}
It is straight forward to verify the average of $\Chat_{n+1}$ over $\zeta^{(k)}_n$ and $\xi^{(k)}_{n+1}$ matches $\Riccati_n(\Chat_n)$, that is, $\E_n \Chat_{n+1}=\Riccati_n(\Chat_n)$. 
 
In order to control the sampling error $\|\Chat_{n+1}-r\Riccati_n(\Chat_n)\|$, it is necessary to have a sufficiently large $K$. Unfortunately, the size of $K$ would need to grow linearly with $d$ \cite{BL08}.  As a simple example,  let $\Chat_n=\Khat_n=0$, $\Sigma_n=I_d$, $r=1$, then $\Delta \Xhat^{(k)}_{n+1}=\xi^{(k)}_{n+1}$ are i.i.d. samples  from $\mathcal{N}(0, I_d)$, and the target sample matrix is $I_d$. Yet $\|\Chat_{n+1}\|=1+\sqrt{d/K}$ with high probability by the Bai-Yin's law \cite{Ver11}. In  practical settings, $K\ll d$, so the sample covariance is unlikely to be correct. 

As discussed in Section \ref{sec:setup}, the main idea of localization is that we assume the target covariance $\Riccati_n(\Chat_n)$  is localized, so it suffices to consider $\Riccati_n(\Chat_n)\circ \bfD_L$, which can be sampled by $\Chat_{n+1}\circ\bfD_L$. Here $\bfD_L$ can be any matrix of form \eqref{eqn:local}, where its radius $L$ does not need to match $l$ used in \eqref{eqn:Ki}. In fact, we will mostly use $\bfD_L=\Dcut^L$ \eqref{eqn:Dcut} with $L\geq 4l$ in our discussion. 
One important advantage gained by localization is that, in order for the covariance sampling to be accurate,  
that is $\|(\Chat_{n+1}-\Riccati_n(\Chat_n))\circ \bfD_L\|$ to be small, 
the necessary sample size scales only with $D_L \log d$, instead of $d$, where $D_L$ is some constant that only depends on $L$. This phenomenon was discovered in statistics \cite{BL08}, assuming the samples are generated from one fixed distribution. But in EnKF, the conditional mean of each sample is different, i.e. $\E_n \Delta \Xhat^{(k)}_{n+1}=\sqrt{r}A_n (I-\Khat_n H)\Delta \Xhat_n^{(k)}$ . A generalization of \cite{BL08} is our first result:

\begin{thm}
\label{thm:localconcen}
For any fixed group of $a_k\in \reals^d$, $k=1,\ldots, K$,  and $K$ i.i.d. samples $z_k \sim \mathcal{N}(0,\Sigma_z)$. Consider the sample covariances
\[
Z=\frac{1}{K}\sum_{k=1}^K (a_k+ z_k)\otimes (a_k + z_k),\quad \Sigma_a=\frac1K\sum_{k=1}^Ka_k\otimes a_k.
\]
Let 
\[
\sigma_{a,z}=\max_{i,j}\{[\Sigma_z]_{i,i}, [\Sigma_a]_{i,i}^{1/2}[\Sigma_z]_{j,j}^{1/2}\}.
\]  
$Z$ concentrates around its mean in the following two ways, where $c$ is an  absolute constant:
\begin{enumerate}[a)]
\item Schur product with a symmetric matrix $\bfD_L$.  For any $t\geq 0$
\[
\Prob(\|(Z-\E Z)\circ \bfD_L\|\geq \|\bfD_L\|_1 \sigma_{a,z} t)\leq  8\exp\left(2\log d -c K \min\{t,t^2\}  \right).
\]
Recall that  $\|\bfD_L\|_1:=\max_{i}\sum_{j=1}^d|[\bfD_L]_{i,j}|$, which is often independent of $d$.
\item Entry-wise. Consider $\|Z-\E Z\|_\infty=\max_{i,j}|[Z-\E Z]_{i,j}|$, then for any $t\geq 0$
\[
\Prob(\|Z-\E Z\|_{\infty}\geq \sigma_{a,z} t)\leq 8\exp\left(2\log d - c  K\min\{t,t^2\}  \right).
\]
\end{enumerate}
\end{thm}
In application to LEnKF, we will let 
\[
a_k=\sqrt{r} A_n (I-\Khat_nH) \Delta \Xhat^{(k)}_n,\quad z_k=\sqrt{r} A_n\Khat_n \zeta^{(k)}_n+ \sqrt{r} \xi^{(k)}_{n+1},
\]
and Theorem \ref{thm:localconcen} shows that $\Chat_{n+1}\circ \bfD_L$ concentrates around $r \Riccati_n(\Chat_n)\circ \bfD_L$. 
 The exact statement is given below by Corollary \ref{cor:concentration}. The  result  in \cite{BL08} is equivalent to the special case where $a_k\equiv 0$. Fortunately, the generalization is not difficult and is in Section \ref{sec:concen}.

\subsection{Localization inconsistency with localized covariance}
While the localization technique makes the covariance sampling much easier, they also introduce additional errors. The fundamental reason is that the localization techniques are applied to the covariance matrices, but cannot be applied to the ensemble members themselves. On the other hand, the analysis update is applied to the  ensemble but not to the covariance. This leads to a matrix inconsistency \cite{WH02, WHWST08, Ner15}.

To illustrate, we look at the forecast filter error at time $n$, $\ehat_n=\overline{\Xhat}_n-X_n$. At this moment, the sample noise realization of $\mathcal{F}^S_{n}$ is available, so it is natural to  consider the conditional covariance of the forecast filter error :
\[
\E_{\mathcal{F}^S_{n} }\ehat_n \otimes \ehat_n=\E_{S}\ehat_n \otimes \ehat_n. 
\]
The identity holds because the sample noises after time $n$ are independent of $\hat{e}_n\in \mathcal{F}_{n}$. 

Suppose this covariance is captured by the localized ensemble covariance, in other words $\E_{S}\ehat_n \otimes \ehat_n=\Chat_n\circ \bfD_L$. 
Based on the LEnKF formulation \eqref{sys:EnKFloc}, the filter errors after the next assimilation step and forecast step are:
\[
e_{n}=\Xbar_{n}-X_{n}=\overline{\Xhat}_{n}-\Khat_{n}(H\overline{\Xhat}_{n}-H X_{n}-\zeta_{n})-X_{n}=(I-\Khat_{n}H) \ehat_{n}+\Khat_{n}\zeta_{n},
\]
\begin{equation}
\label{tmp:errorupdate}
\ehat_{n+1}= \overline{\Xhat}_{n+1}-X_{n+1}
=A_n (\Xbar_n-X_n)-\xi_{n+1}=A_n(I-\Khat_{n}H) \ehat_{n}+A_n \Khat_n\zeta_n -\xi_{n+1}.
\end{equation}
Since the Kalman gain $\Khat_n\in \mathcal{F}^S_n$, $\zeta_n$ and $\xi_{n+1}$ are independent of $\mathcal{F}^S_\infty$,  the new forecast error covariance  is 
\begin{align}
\notag
\E_{S}  &\ehat_{n+1}\otimes \ehat_{n+1}= A_n[(I-\Khat_{n}H) (\E_S \ehat_{n}\otimes \ehat_{n}) (I-\Khat_{n}H)^T+\sigma_o^2\Khat_{n}\Khat_{n}^T]A_n^T+\Sigma_n,\\
\label{tmp:realerror}
&= A_n[(I-\Khat_{n}H) [\Chat_n\circ \bfD_L] (I-\Khat_{n}H)^T+\sigma_o^2\Khat_{n}\Khat_{n}^T]A_n^T+\Sigma_n=:\Riccati_n'(\Chat_n).
\end{align}
On the other hand, the ensemble covariance is generated by the update in \eqref{tmp:Delta}. With no inflation, $r=1$, Theorem \ref{thm:localconcen} indicates $\Chat_{n+1}\circ \bfD_L$ is near its average 
\begin{equation}
\label{tmp:ensemblerror}
\Riccati_n(\Chat_n)\circ\bfD_L=[A_n [(I-\Khat_{n} H) \Chat_n (I-\Khat_{n} H)^T+\sigma_o^2\Khat_{n}\Khat_n^T]A_n^T+\Sigma_n]\circ \bfD_L. 
\end{equation}
Recall the posterior Riccati map $\Riccati_n(\Chat_n)$ is defined by \eqref{eqn:Riccati}.

The difference between \eqref{tmp:realerror} and \eqref{tmp:ensemblerror} can be interpreted as the inconsistency caused by commuting the localization and Kalman covariance update.  In order for the ensemble covariance to capture the error covariance, it is necessary for this difference to be small. This is an issue not governed by the sampling scheme, but governed by the localization operation.

As discussed in the introduction, the major motivation behind localization techniques is that  the covariance is  localized. We formalize  this notion through the following definition. 
\begin{defn}
\label{defn:localstructure}
Given a decreasing function $\Phi: \reals^+\mapsto [0,1]$ with $\Phi(0)=1$, we say the forecast covariance sequence $\Chat_n$ follows an  $(M_n, \Phi, L)$-localized structure, if 
\begin{equation}
\label{eqn:localstructure}
|[\Chat_n]_{i,j}|\leq \begin{cases} M_n \Phi (\dist (i,j)) \quad & \dist(i,j)\leq L;\\
M_n\Phi(L)\quad & \dist(i,j)>L.
\end{cases}
\end{equation}
\end{defn}
The decay function $\Phi$ and $L$ need not coincide with the $\phi$ and $l$ used in Kalman gain localization \eqref{eqn:local}. This flexibility is useful when we try to verify the localized structure. 
Intuitively, in order for localization techniques to be effective, we  need $\Phi(x)$ to be near zero when $x$ is large. This holds true for most localized covariance structures, such as the Gaspari Cohn matrix \eqref{eqn:GC}, and also the function $\Phi(x)=\lambda_A^{x}$ with a certain $\lambda_A<1$, which will appear below in Theorem \ref{thm:formloc} for linear systems. 

One interesting phenomenon, is that if the forecast covariance is already localized, then the localization inconsistency is in general small:
\begin{prop}
\label{lem:Kalmanlocalizationerror}
Suppose $\|A_n\|\leq M_A$, $A_n$ and $\Sigma_n$ are of bandwidth less than $l$, and $\Chat_n$ follows an $(M_n,\Phi, L)$-localized structure, then the localization inconsistency with $\bfD_L=\Dcut^L$ and $L\geq 4l$, given by
\[
\Delta_{loc}=\eqref{tmp:realerror}-\eqref{tmp:ensemblerror},
\]
has nonzero entries only around the localization boundary:
\[
[\Delta_{loc}]_{i,j}=0\quad \text{if}\quad |\dist(i,j)-L|> 2l.
\]
Moreover, it is bounded by 
\begin{equation}
\label{eqn:localincon}
\|\Delta_{loc}\|\leq  M_n M_A^2 (1+ \sigma_o^{-2} \calB_l M_n)^2 \calB_l^2\calB_{L,l} \Phi(L-2l).
\end{equation}
$\calB_{L,l}$ is a volume constant $\calB_{L,l}=\max_i \#\{j: |\dist(i,j)-L|\leq 2l\}$, and $\calB_l$ is given by \eqref{eqn:calB}.  Note that if $\Phi(L-2l)$ is close to zero, the right side is very small. 
\end{prop}

%
%
%
\subsection{Main result: LEnKF performance}
There are different ways to quantify the performance of EnKF. 
One approach is to compare EnKF with its large ensemble limit, 
which is the Kalman filter, and estimate the convergence rate \cite{LeG11, mandel2011convergence, KM15, LTT16}. 
Moreover, advanced sampling techniques, such as multilevel Monte Carlo, can be applied to the EnKF procedures, 
and speed up the convergence \cite{HLT16, CHLT16}.
However, these results have not investigated the dependence of  sample size $K$ on the underlying dimension,
thus they are not helpful in explaining the advantages of the localization procedures.
Moreover, the large ensemble limit for LEnKF is not necessarily the optimal, since the localization techniques may violate the Bayes' formula. 

A more practical approach looks for qualitative EnKF properties, where the necessary sample size $K$ scales with quantities much less than $d$ \cite{KLS14, TMK15non,TMK15, KMT15}, for example a low effective dimension \cite{MT17cpam}.  One central issue of EnKF is that, unlike Kalman filter, 
it estimates the forecast uncertainty by the ensemble covariance, which can be faulty.
Since the forecast covariance matrix plays a pivotal role in the EnKF operation, 
it is important to ask if the ensemble covariance captures the real filter error covariance. 

In our particular case, we are interested in finding a bound for filter error covariance $\E_S \ehat_n \otimes \ehat_n$.  
We will compare it with the filter ensemble covariance $\Chat_{n}$.  
Note that the conditioning $\E_S$ is with respect to the sample noise filtration $\mathcal{F}^S_\infty$ given in \eqref{eqn:filterationS}, 
moreover note that $\Chat_{n}\in \mathcal{F}_\infty^S$. 
Therefore the comparison is legitimate. 
By showing $\E_{S} \ehat_n \otimes \ehat_n$ is dominated by a proper inflation of $\Chat_n$ with large probability, 
we demonstrate that the LEnKF reaches its estimated performance. 
In order to achieve that, we need  the localized structure to be stable as well.

\begin{thm}
\label{thm:main}
Suppose the forecast ensemble covariance follows a stable $(M_n,\Phi,L)$-localized structure, and the sample size $K$ exceeds $D_L \log d$ with  a constant $D_L$ that depends on $L$, the LEnKF \eqref{sys:EnKFloc} reaches its estimated performance in the long time average. In specific, for any $\delta>0$, suppose the following conditions hold 
\begin{enumerate}[1)]
\item In the signal-observation system \eqref{sys:signalobs},  $A_n$ and $\Sigma_n$ are of  bandwidth $l$, moreover 
\[
\|A_n\|\leq M_A,\quad m_\Sigma I_d\preceq \Sigma_n\preceq M_\Sigma I_d,\quad M^2_A\geq m_\Sigma.
\] 
\item Suppose the initial error satisfies  $\E_S \ehat_0 \otimes \ehat_0\preceq r_0(\Chat_0+\rho I_d)$ for some  $r_0$ and $\rho$ that
\[
0<r_0,\quad 0<\rho< (\tfrac12-\tfrac{1}{2r})\min\{M_A^2/m_\Sigma, \sigma_o^2\}.
\]
This can always be achieved by picking a larger $r_0$. 
\item The forecast covariance $\Chat_n$ follows a $(M_n,\Phi,L)$-localized structure as in Definition \ref{defn:localstructure}. Moreover, the localized structure is stable, so there are constants $B_0,D_0$ and $M_0$ so that 
\begin{equation}
\label{eqn:stabloc}
\frac{1}{T}\E \sum_{n=1}^T M_n\leq \frac{1}{T}(B_0 \E \|\Chat_0\|+D_0) + M_0. 
\end{equation}
\item The localized structure $\Phi$ and radius $L$ satisfy 
\[
L\geq 4l,\quad \Phi(L-2l)\leq \delta^3 \calB_{L,l}^{-1}M^{-2}_A \calB_l^{-6}.
\]
The volume constants are given by Proposition \ref{lem:Kalmanlocalizationerror}.
\item The sample size $K>\Gamma(r\calB_l\delta^{-1},d)$, with
\begin{equation}
\label{eqn:Gamma}
\Gamma(x,d)=\max\{9x^2, \tfrac{24 x}{ c}, \tfrac{18 x^2}{c} \log d\},
\end{equation}
and the absolute constant $c$ is given by Theorem \ref{thm:localconcen}. 
\end{enumerate}
Then for any $1<r_*<r$, the filter error covariance is dominated by the filter covariance 
\[
\E_{S} \ehat_n \otimes \ehat_n\preceq r_*(\Chat_n\circ \Dcut^L+\rho I_d)
\]  
with high $1-O(\delta)$ probability  in long time average
\begin{align*}
1-\frac{1}{T}\sum_{n=0}^{T-1} &\Prob(\E_{S} \ehat_n \otimes \ehat_n\preceq r_*( \Chat_n\circ \Dcut^L+\rho I_d))\\
&\leq \frac{r_0}{T \log r_*}+\frac{\delta (B_0\|\Chat_0\|+D_0)}{T \log r_*}(\rho^{-1} \calB_l^2 M_A^2+\tfrac{2r^{1/3}}{(\rho\sigma_o)^{1/3}})\\
&\quad +\frac{\delta}{\log r_*}\left((\rho^{-1} \calB_l^2 M_A^2+\tfrac{2r^{1/3}}{(\rho\sigma_o)^{1/3}})M_0+\rho^{-1} M_\Sigma+2\tfrac{r^{1/3}}{\rho^{1/3}}\sigma^{2/3}_o  \right).
\end{align*}
\end{thm}

\subsection{Weak  local interaction with sparse observations}
By Theorem \ref{thm:main}, the stability of localized structure is a necessary condition for  the LEnKF to reach its estimated performance. While in practice this condition is often assumed to be true to motivate the localization technique, and one can check it while the algorithm is running, it is interesting to find some sufficient a-priori conditions of system \eqref{sys:signalobs}, so that \eqref{eqn:stabloc} holds. Unfortunately,  rigorous investigations in this direction is sorely missing. Here we provide a stability analysis in a simple setting. 

The origin of localized covariances is intuitively clear.  In most physical systems, the covariance between  $[X]_i$ and $[X]_j$ comes from information propagation in space. So if the propagation is weak and decays at the same time, there will be a localized covariance. For our linear models, the information propagation is carried by local interactions, described by the off diagonal terms of $A_n$. To enforce its weakness, we assume that there is 
a $\lambda_A<1$, such that  
\begin{equation}
\label{aspt:weakinter}
\max_{i}\left\{\sum_{k=1}^d |[A_n]_{i,k}|\lambda_A^{-\dist(i,k)}\right\} \leq \lambda_A. 
\end{equation}
For the simplicity of our discussion, we also assume the system noise is diagonal  $\Sigma_n=\sigma^2_\xi I_d$.

Note that $\lambda_A<1$, so $\lambda_A^{-\dist(i,k)}$ is a large number when $i$ and $k$ are fart apart. So condition \eqref{aspt:weakinter} constraints the long distance interaction, measured by $|[A_n]_{i,k}|$, to be weak. In other words, \eqref{aspt:weakinter} models a local interaction.  If we concern the unfilter covariance of the sequence $[X]_i$, then $\lambda_A<1$ is sufficient to guarantee the covariance is localized, using Proposition \ref{prop:localforecast} in below.

The main difficulty actually comes from the observation part. For simplicity, we require the observations in \eqref{sys:observation} to be sparse in the sense that $\dist(o_i, o_j)>2l$ for any $i\neq j$. Recall that $o_i$ is the $i$-th observable component. Then for each location $i\in \{1,\cdots, d\}$, there is at most one location $o(i)\in\{o_1,\cdots, o_q\}$ such that $\dist(i,o(i))\leq l$. This will significantly simplify the analysis step and yield an explicit expression. Sparse observations are in fact quite common in practice. Moreover, it is also possible to generalize the results here to non-sparse scenario, by using sequential assimilation \cite{Ner15}. But the conditions will be much more involved.

Under the sparse observation scenario, the following function describes how does the localized structure of $\Chat_n$ update to the one of $\Chat_{n+1}$:
\begin{equation}
\label{eqn:psi}
\psi_{\lambda_A}(M,\delta)=(r+\delta)\max\left\{\lambda_A M\left(1+\sigma_o^{-2}M\right)^2+\lambda_A \sigma_o^{-2}M^2, \lambda^2_A M+\sigma_\xi^2\right\}.
\end{equation}
This  function provides a way to ensure stable localized structure:
\begin{thm}
\label{thm:formloc}
Given a LEnKF \eqref{sys:EnKFloc}, suppose the following holds 
\begin{enumerate}[1)]
\item The system noise is diagonal  and the observations are sparse
\[
\Sigma_n=\sigma_\xi^2 I_d,\quad \dist(o_i, o_j)>2l,\,\,\,\forall i\neq j. 
 \]
\item There is a $\lambda_A<r^{-1}$ such that \eqref{aspt:weakinter} holds.
\item There are constants
\[
 0<\delta_*<\min\{0.25, \tfrac12(\lambda_A^{-1}-r)\},\quad M_*\geq \frac{(r+2\delta_*)\sigma_\xi^2}{1-\lambda_A},
\]
 such that $\psi_{\lambda_A}(M_*, \delta_*)\leq M_*$ with $\psi_{\lambda_A}$ given by \eqref{eqn:psi}. 
\item Denote $n_*=2L+\lceil\frac{\log 4\delta_*^{-1}}{\log \lambda^{-1}_A}\rceil$.  The sample size $K$ exceeds 
\begin{equation}
\label{tmp:sampleK}
K>\max\left\{-\frac{1}{c\delta_*^2 \lambda_A^{2L}}\log (16d^2 n_* \delta_*^{-2}), \Gamma(2r\delta_*^{-1},d) \right\}.
\end{equation}
\end{enumerate}
Then the forecast ensemble covariance follows a stable localized structure $(M_n, \Phi, L)$ with $\Phi(x)=\lambda_A^{x}$. In specific, the stochastic sequence $M_n$ is dissipative every $n_*$ steps:
\[
\E_0 M_{n_*}\leq \frac{1}{2} M_{0}+ (1+2\delta_*)M_* \,.
\]
The long time average condition \eqref{eqn:stabloc} can be verified by 
\[
\frac{1}{T}\sum_{k=1}^T\E  M_{k}\leq \frac{2n_* }{T\lambda_A^{L}}(\E \|\Chat_0\|+M_*)+2(1+\delta_*)M_*.
\]
\end{thm}
\begin{rem}
Note that 
\[
\psi_{\lambda_A}(M,0)=\max\{r\lambda_A M(1+\sigma_o^{-2}M)^2+r\lambda_A\sigma_o^{-2}M^2,r\lambda_A^2 M+r\sigma_\xi^2\},
\]
With sufficiently small $\lambda_A$ or $\sigma^{-1}_o$, $\psi_{\lambda_A}(M,\delta)<M$ can have a solution, so condition $3)$ holds.
\end{rem}
\subsection{Localization radius}
\label{sec:radius}
One important and difficult issue of LEnKF implementation is how to choose the localization radius $l$. The theoretical results above shed some light over this issue qualitatively. It is worth noticing that this paper has two localization radii. $l$ is the one used for LEnKF\eqref{sys:EnKFloc} formulation, and $L$ is used for the filter error theoretical analysis. But generally speaking $L$ and $l$ should be picked so that $L\geq 4l$, so we concern only of $L$ in the following. We also assume that $\Phi(x)=\lambda_A^x$ from Theorem \ref{thm:formloc} for simpler discussion. 

A smaller localization radius simplify the sampling task by focusing  on a smaller assimilation domain, and significantly reduces the necessary sample size. This comes from two perspectives. First, in order for the LEnKF to sample the correct localized covariance matrix, condition 5) of Theorem \ref{thm:main} requires the sample size to grow polynomially with $L$, since $\|\Phi\|_1$ is summing over $\calB_L$ entries. Second, the localized covariance structure can be very delicate at the boundary, and to maintain it one needs the random forecast covariance to have sampling error of scale $\lambda_A^{L}$. This leads to the exponential dependence of $K$ on $L$, as in condition 4) of  Theorem \ref{thm:formloc}. 

On the other hand, a larger localization radius $L$ reduces the size of the localization inconsistency. Based on Proposition \ref{lem:Kalmanlocalizationerror}, the localization inconsistency is of order $\Phi(L-2l)=\lambda_A^{L-2l}$, because within inequality \eqref{eqn:localincon}, $\calB_l$ is independent of $L$, and $\calB_{L,l}$ is also independent of $L$ if $i,j$ are taken from $\{1,\cdots,d\}$. This becomes condition 4) of Theorem \ref{thm:main}, where we need the localization radius to be large, so the inconsistency is bounded by the tolerance.

\subsection{LEnKF accuracy with  small noises}
In practice, with frequent and accurate observations, the system noises, $\Sigma_n$ and $\sigma_o^2$, are often of scale $\epsilon$. In this scenario, the LEnKF has its error covariance scale with $\epsilon$ in long time, showing an accurate forecast skill. Moreover, there is no requirement that the initial ensemble to have error of scale $\epsilon$, meaning the LEnKF can converge to the signal $X_n$ given enough time. 
\begin{thm}
\label{thm:accuracy}
Suppose, the signal-observation system \eqref{sys:signalobs} satisfies the conditions of Theorem \ref{thm:formloc}, and its LEnKF is tuned to satisfy the conditions of Theorem \ref{thm:main} except \eqref{eqn:stabloc}. Then if the same LEnKF is applied to the following system
\[
\begin{gathered}
 X^\epsilon_{n+1}=A_n X^\epsilon_n+b_n +\xi_n,\quad \xi_{n+1}\sim \mathcal{N}(0, \epsilon\sigma_\xi^2 I_d),\\
 Y^\epsilon_{n+1}=H X^\epsilon_{n+1}+\zeta_n,\quad \zeta_{n+1}\sim \mathcal{N}(0, \epsilon \sigma_o^2 I_q),
\end{gathered}
\]
it has small filter error covariance of scale $\epsilon$. In particular, the ensemble  covariance is of scale $\epsilon$ in long time average
\[
\frac{1}{T}\sum_{n=1}^T\E  \|\Chat_n\|_\infty\leq \frac{2n_* }{T\lambda_A^{L}}(\E \|\Chat_0\|+\epsilon M_*)+2(1+\delta_*)\epsilon M_*.
\]
Moreover, the real filter covariance is dominated by $\Chat_n$ with high probability:
\begin{align*}
1-\frac{1}{T}\sum_{n=0}^{T-1} &\Prob(\E_{S} \ehat_n \otimes \ehat_n\preceq r_*( \Chat_n\circ \Dcut^L+\epsilon \rho I_d))\\
&\leq \frac{r_0}{T \epsilon \log r_*}+\frac{2\delta n_* (\E \|\Chat_0\|+\epsilon M_*)}{T \epsilon \lambda_A^{L} \log r_*}(\rho^{-1} \calB_l^2 M_A^2+\tfrac{2r^{1/3}}{(\rho\sigma_o)^{1/3}})\\
&\quad +\frac{\delta}{\log r_*}\left(2(\rho^{-1} \calB_l^2 M_A^2+\tfrac{2r^{1/3}}{(\rho\sigma_o)^{1/3}})(1+\delta_*) M_*+\rho^{-1} \sigma_\xi^2+2\tfrac{r^{1/3}}{\rho^{1/3}}\sigma^{2/3}_o  \right).
\end{align*}
Note that $\epsilon$ appears only in terms that converge to zero with $T\to \infty$. 
\end{thm} 

\begin{rem}
We need the system to follow the conditions in Theorem \ref{thm:formloc} only to ensure the stable localized structure exists. If one can find other conditions to verify that the LEnKF follows an $(M_n, \Phi, L)$ localized structure  such that $M_n$ converges to a scale of $\epsilon$, the conditions in Theorem \ref{thm:formloc} can be replaced. 
\end{rem}

\section{Numerical experiments}
\label{sec:num}
There is plenty of numerical evidence showing that LEnKF has good forecast skill even with nonlinear dynamical systems. Moreover, this paper intends to understand LEnKF from a theoretical perspective, not an empirical one. On the other hand,  several new concepts and conditions are introduced in our analysis framework. To understand their significance, we conduct a few simple numerical experiments in this section.

\subsection{Experiments setup: a stochastic turbulence model}
We consider a stochastically forced dissipative advection equation on an one dimensional periodic domain from Section 6.3 of \cite{MH12}:
\[
\frac{\partial u(x,t)}{\partial t}=c \frac{\partial u(x,t)}{\partial x}-\nu u(x,t)+\mu \frac{\partial^2 u(x,t)}{\partial x^2}+\sigma_x \dot{W}(x, t).
\]
To transform it to a discrete linear system, we apply the centered difference formula with spatial grid size $h$, and Euler scheme with time step $\Delta t$. We assume $W(x,t)$ is a white noise in both time and space. The discretized signal-system $[X_{n,1},\cdots, X_{n,d}]^T$ follows
\begin{equation}
\label{sys:PDE}
\begin{gathered}
X_{n+1,i}=a_- X_{n,i-1}+a_0 X_{n,i}+a_+ X_{n,i+1}+ \sigma_x\sqrt{\Delta t} W_{n+1,i}, \quad i=1,\ldots,d;\\
a_-=\tfrac{\mu \Delta t}{h^2}-\tfrac{c \Delta t}{2h},\quad a_0=1-\tfrac{2\mu \Delta t}{h^2}-\nu\Delta t,\quad a_+=\tfrac{\mu \Delta t}{h^2}+\tfrac{c \Delta t}{2h}.
\end{gathered}
\end{equation}
The indices should be interpreted cyclically, that is $X_{n,0}=X_{n,d}$ and $X_{n,d+1}=X_{n,1}$. 
The natural distance between indices is $\dist(i,j)=\min\{|i-j|,||i-j|-d|\}$. The system noises $W_{n,i}$ are independent samples from $\mathcal{N}(0,1)$. 
We also initialize $X_{0,i}\sim \mathcal{N}(0,1)$ for simplicity. 
Evidently, if we formulate \eqref{sys:PDE} in the format of \eqref{sys:signalobs}, the corresponding matrix $A_n$ is constant with  bandwidth $l=1$. In other words it is tridiagonal.
We assume one observation is made every $p$ components with independent Gaussian noise $B_{n,k}\sim \mathcal{N}(0,1)$:
\[
Y_{n,k}=X_{n,p(k-1)+1}+\sigma_o B_{n,k}.
\]

A simple LEnKF with domain localization radius $l=1$, inflation $r=1.1$ will be applied to recover $X_n$. A small sample size $K=10$ is taken. As comparison, we implement  a standard EnKF with the same inflation, sample size and sample noise realization. A standard Kalman filter is also computed to indicate the optimal filter error. We are interested to see
\begin{itemize}
\item Does LEnKF have a close to optimal performance? Does localization play a key role?
\item Is filter performance robust against dimension increase?
\item Does filter performance scale with the noise strength?
\item Does the LEnKF ensemble covariance localize, and is this structure stable?
\item Do the a-priori conditions of Theorem \ref{thm:formloc} hold?
\end{itemize}

In the discussion below, we consider dimension in a wide range $d=10,100,1000$. Yet we will fix the grid size $h$ in each regime. This corresponds to a sequence of domains with increasing size, but not a fixed domain with increasing refinement. Although the latter  can also have very high dimension, localization is not a suitable tool; a proper projection to the low effective dimension should be more effective \cite{MT17cpam}.
Also it is worth noticing that there are better ways to filter \eqref{sys:PDE}, such as Fourier domain filtering \cite{MH12}. We are running LEnKF here just to support our theoretical analysis.

\subsection{Regime I: strong dissipation}
We first consider a regime of \eqref{sys:PDE} with strong uniform damping and weak advection
\[
h=1,\quad \Delta t=0.1,\quad p=5, \quad\nu=5,\quad c=0.1,\quad \mu=0.1,\quad \sigma_x=\sigma_o=1.
\]

 In this regime, the conditions of Theorem \ref{thm:formloc} can be verified. In particular, \eqref{aspt:weakinter} can be formulated as
\begin{equation}
\label{tmp:tridiagonal}
a_- \lambda_A^{-1} +a_0+a_+ \lambda_A^{-1}\leq \lambda_A. 
\end{equation}
Direct numerical computation shows that $\lambda_A=0.5186$ satisfies this relation. Furthermore, we can verify that $(\delta^*, M_*)=(0.128,0.2187)$ satisfy condition 3) of Theorem \ref{thm:formloc}. Theorem \ref{thm:formloc} predicts a stable stochastic sequence $M_n$ exists so $\Chat_n$ follows localized structure $(M_n, \Phi, 4)$, where $\Phi(x)=\lambda_A^{x\wedge L}$ and $M_n$ has its mean bounded by $8.8959$. On the other hand,  Theorem \ref{thm:formloc} requires the  sample size to be around $K=2.8\times 10^4$ for $d=100$, and $K=7.34\times 10^4$ for $d=10^6$. We will see $K=10$ is sufficient for LEnKF to perform well numerically. The overestimate is reasonable as  theoretical analysis is often too conservative. The main point of theoretical analysis is showing a logarithmic dependence of $K$ on the dimension.

The numerical results are presented in Figure \ref{fig:stable}. In subplot a) the dimension average square forecast error 
\[
\text{DSE}:=|X_n-\overline{\Xhat}_n|^2/d
\]
 of LEnKF is plotted for 100 iterations. The time mean DSE (MSE) is around 0.142 for $d=100$. This is comparable with the  optimal Kalman filter MSE 0.129. Moreover, this performance is robust for all dimensions, MSE=0.137 for $d=10$ and MSE=0.143 for $d=1000$, while the oscillation is stronger in $d=10$ case due to averaging over a small dimension. 
 
 Since this regime is very stable, EnKF without localization also has surprisingly good performance, as shown in subplot b). Its MSE is around 0.15, which is worse than LEnKF. This shows that, while the conditions of Theorem \ref{thm:formloc} are sufficient for LEnKF to work well, they might be too strong. It will be interesting if sharper working conditions for LEnKF can be found. It will also be interesting if one can show such strong conditions can already guarantee EnKF to work without localization. 
   
      Two other properties predicted by our theory are also validated. In subplot c), the localization status $M_n$ is plotted for all three dimensions. All three time sequences are stable, and they are all bounded below the theoretical estimate $8.8959$ from Theorem \ref{thm:formloc}. We also test LEnKF with small scale system noises $\sigma^{\epsilon}_x=\sqrt{\epsilon} \sigma_x,\sigma^{\epsilon}_o=\sqrt{\epsilon} \sigma_o$. In subplot d), we plot the time mean DSE of $\epsilon=1,\frac12, \frac{1}{4},\ldots, \frac{1}{32}$ in logarithmic scales. It is clear that the LEnKF has the correct MSE scale of $\epsilon$ as Theorem \ref{thm:accuracy} predicted.
  \begin{figure}
  \centering
  \hspace*{-2.5cm}
\includegraphics[scale=0.45]{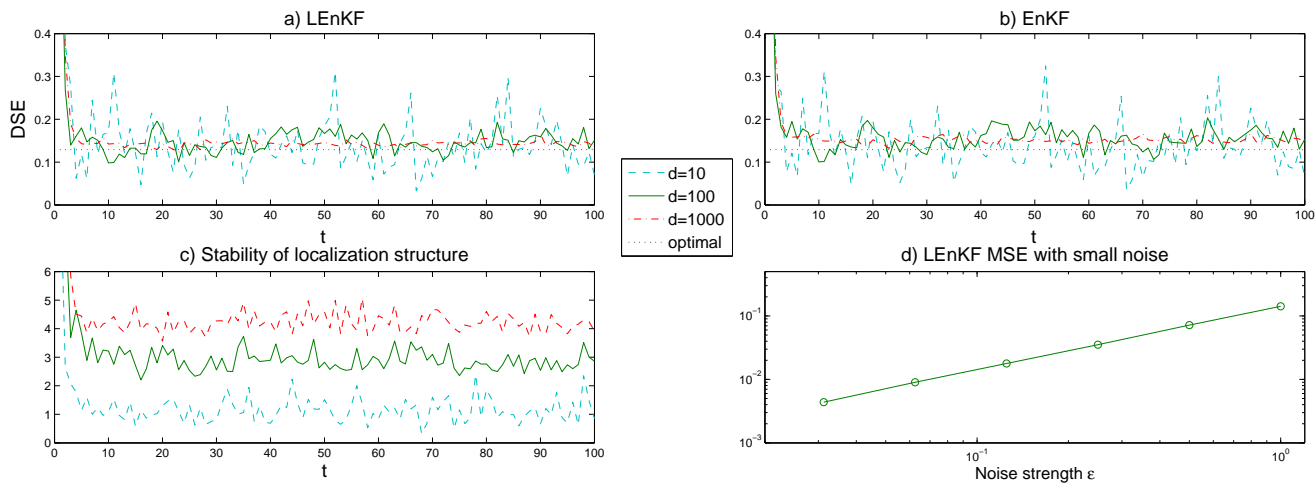}
\caption{\label{fig:stable} Filter performance in stable regime I. }
\end{figure}

\subsection{Regime II: strong advection}
The second regime we considered has a strong advection, while the damping is weak:
\[
h=0.2,\quad \Delta t=0.1,\quad p=5, \quad\nu=0.1,\quad c=2,\quad \mu=0.1,\quad \sigma_x=\sigma_o=1.
\]
This regime is close to unstable, since the linear system map $A_n$ has spectral norm 0.99. \eqref{tmp:tridiagonal} does not have a solution below $1$, so the conditions of Theorem \ref{thm:formloc} are not verifiable. Nevertheless, we find empirically the LEnKF ensemble covariance matrices are localized. In Figure \ref{fig:cov}, we demonstrate this by plotting 
\[
\widehat{\Phi}(x)=\frac{1}{d}\E \left(\sum_{i=1}^{d-x} |[\Chat_{n}]_{i,i+x}|+\sum_{i=d-x+1}^{d} |[\Chat_{n}]_{i,i+x-d}|\right)
\]
using empirical average from 1000 samples with $d=100, n=100.$ The clear covariance strength transition around $x=4$ indicates that the ensemble covariance is localized. Therefore Theorem \ref{thm:main} applies and predicts that LEnKF will have a good performance.
\begin{wrapfigure}{R}{0.5\textwidth}
\centering
\includegraphics[width=0.6\textwidth]{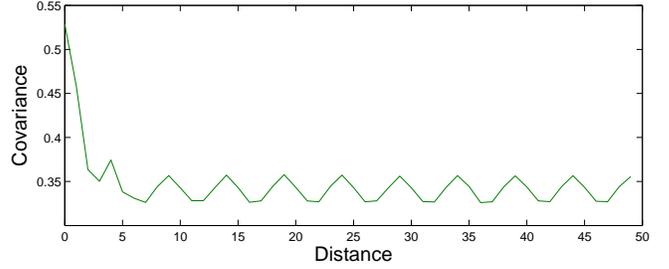}
\caption{\label{fig:cov} Localization structure}
\end{wrapfigure}

This is indeed the case. In subplot a) of Figure \ref{fig:advec}, we see that LEnKF has a forecast skill. The MSE is around 1.63 for $d=100$, where the optimal Kalman filter MSE is 1.06. This performance does not change much with the dimension, MSE=1.42 for $d=10$, MSE=1.72 for $d=1000$. The EnKF on the other hand is highly unstable except for the low dimension $d=10$ case. In subplot b), we see for $d=100$ and $1000$, the DSE of EnKF grows exponentially to $10^{10}$. This is a phenomenon known as EnKF catastrophic filter divergence, previously studied by \cite{MH12, KMT15}. Now this also demonstrates how important is the localization technique.  Such divergence can be resolved by introducing an adaptive additive inflation, where the stability can be rigorously proved \cite{TMK15}. 

In this unstable regime,  LEnKF retains its stability and accuracy. Since the localization structure does not have a theoretical ground in this regime,  Figure subplot c) plots only the largest matrix component of $\Chat_n$. From it we see the LEnKF ensemble covariance is stochastically stable for all three dimensions. Like in Regime I, we also test  LEnKF with small scale system noises $\sigma^{\epsilon}_x=\sqrt{\epsilon} \sigma_x,\sigma^{\epsilon}_o=\sqrt{\epsilon} \sigma_o,$ where $\epsilon=1,\frac12, \frac{1}{4},\ldots, \frac{1}{32}$. Subplot d) indicates the LEnKF has the correct MSE scaling with $\epsilon$.

  \begin{figure}
  \centering
\hspace*{-2.5cm}
\includegraphics[scale=0.45]{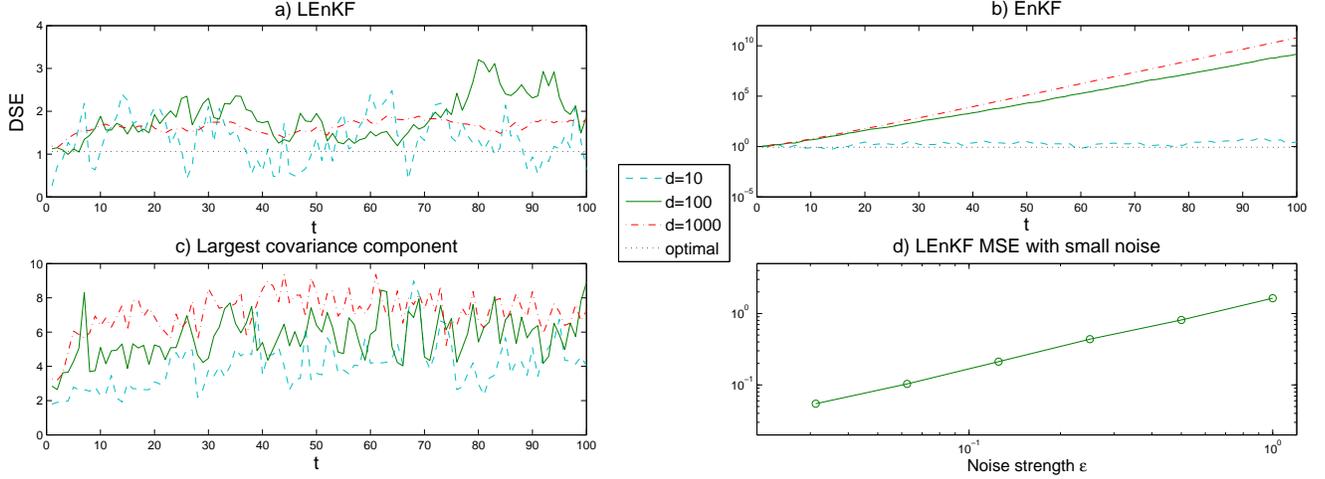}
\caption{\label{fig:advec} Filter performance in stable regime II. }
\end{figure}

\section{Concentration of localized random matrices}
\label{sec:concen}
In this section, we present the proof of Theorem \ref{thm:localconcen}. While part a) is more useful, it can be established easily from part b), using a similar argument as in \cite{BL08}. 
\subsection{Entry-wise concentration}
It is well known that the averages of independent Gaussian variables concentrate around their expected values. In specific,  a simplified version of theorem 1.1 from \cite{RV13} is:
\begin{thm}[Hanson-Wright inequality]
\label{thm:HW}
Let $\xi \sim \mathcal{N}(0,I_n)$ and $A$ be an $n\times n$ matrix. Then  for any $t\geq 0$
\[
\Prob\left( |\xi^T A \xi- \E \xi^T A \xi |>t\right)\leq 2\exp \left(-c\min\left(\frac{t^2}{ \|A\|^2_{HS}}, \frac{t}{\|A\|} \right) \right). 
\]
Here $c$ is a constant independent of other parameters.  The Hilbert-Schmidt (Frobenius) norm is denoted by $\|A\|_{HS}=[\sum_{i,j} [A]_{i,j}^2]^{1/2}$.
\end{thm}
This provides us a straight forward way to control  the random matrix  entries $[Z]_{i,j}$ in Theorem \ref{thm:localconcen}.

\begin{lem}
\label{lem:Fuv}
Under the conditions of Theorem \ref{thm:localconcen},  let $\Delta=Z-\E Z$. There is an absolute constant $c$ such that for any $t\geq 0$, 
\[
\Prob(|[\Delta]_{i,j}|>\sigma_{a,z} t)\leq 8\exp(-c K\min\{t,t^2\}). 
\]
\end{lem}
\begin{proof}
For any vector $u$, denote $\Delta_u=u^T[Z-\E Z] u $. Then by symmetry,
\[
[\Delta]_{i,j}=\frac14(\Delta_{\bfe_i+\bfe_j}-\Delta_{\bfe_i-\bfe_j}). 
\]
Recall that $\bfe_i$ is the $i$-th standard basis vector.
So it suffices to find a concentration bound for $\Delta_u$ with $u=\bfe_i\pm \bfe_j$. To do that, note that $u^T\Sigma_z u=\E u^T z_k z_k^T u$, so we can decompose $\Delta_u$
\begin{align*}
\Delta_{u}&=K^{-1}\sum_{k=1}^K\left[ \langle u, a_k+z_k\rangle \langle u, a_k+z_k\rangle-\langle u, a_k\rangle \langle u, a_k\rangle- \E \langle u, z_k\rangle \langle u, z_k\rangle\right]\\
&=2K^{-1}\sum_{k=1}^K \langle u, a_k\rangle \langle u, z_k\rangle+ K^{-1}\sum_{k=1}^K(\langle u, z_k\rangle \langle u, z_k\rangle- \E \langle u, z_k\rangle \langle u, z_k\rangle )
\end{align*}
We  denote $\langle a, b\rangle=a^Tb$ as the inner product, and the two summations above as I and II in the following. Notice that $\langle u, z_k\rangle\sim \mathcal{N}(0, u^T\Sigma_z u  ),K^{-1}\sum_{k=1}^K \langle \bfe_j, a_k\rangle^2=u^T\Sigma_a u$. Moreover for $u=\bfe_i\pm \bfe_j$, 
\begin{equation}
\label{tmp:usigma}
u^T\Sigma_z u=[\Sigma_z]_{i,i}+[\Sigma_z]_{j,j}\pm 2 [\Sigma_z]_{i,j} \leq 2 ([\Sigma_z]_{i,i}+[\Sigma_z]_{j,j})\leq 4\sigma_{a,z}.
\end{equation}
We have the same conclusion for $u^T \Sigma_a u$. Because $ \langle u, a_k\rangle$ is a deterministic scalar, 
\[
\langle u, a_k\rangle \langle u, z_k\rangle\sim \mathcal{N}(0,  u^T a_k a_k^Tu\cdot u^T\Sigma_z u)
\]
and
\[
\text{I}=2K^{-1}\sum_{k=1}^K \langle u, a_k\rangle \langle u, z_k\rangle\sim \mathcal{N}(0, 4K^{-1} u^T \Sigma_au\cdot u^T\Sigma_z u). 
\]
Because by definition of $\sigma_{a,z}$,  $u^T \Sigma_au\cdot u^T\Sigma_z u\leq 16 \sigma_{a,z}^2 $,  by the Chernoff bound for Gaussian distributions, there is a $c_1>0$ so that
\[
\Prob(|\text{I}|>\tfrac12 \sigma_{a,z}  t)\leq 2 \exp(-  c_1 K t ). 
\]
In order to deal with II,  notice that 
\[
\xi:=\frac{1}{\sqrt{ u^T \Sigma_z u}}[\langle u, z_1\rangle,\cdots, \langle u, z_K\rangle]^T\sim \mathcal{N}(0, I_{K}).
\]
So 
\[
\text{II}=K^{-1}\sum_{k=1}^K(\langle u, z_k\rangle^2- \E \langle u, z_k\rangle^2 )=\xi^T A \xi-\E \xi^T A\xi,
\]
where $A=\frac 1K (u^T \Sigma_z u) I_K $. Clearly, $\|A\|\leq \frac{4\sigma_{a,z}}{K}$, and $\|A\|^2_{HS}\leq \frac{16\sigma_{a,z}^2}{K}$.
Therefore, by Theorem \ref{thm:HW} there is a  constant $c_2$ so that for all $s\geq 0$
\[
\Prob(|\text{II}|> \tfrac 12  s)\leq 2\exp(-c_2 \min K\{\tfrac{ s^2}{\sigma_{a,z}^2}, \tfrac{s}{\sigma_{a,z}}\}).
\]
Let $t=\sigma^{-1}_{a,z} s$, the inequality can be written as
\[
\Prob(|\text{II}|> \tfrac 12  \sigma_{a,z}t)\leq 2\exp(-c_2 K\min \{t, t^2\}).
\]
Because $|\Delta_{u}|\leq |\text{I}|+|\text{II}|$, by the union bound, if we let $c=\min\{c_1, c_2\}$ ,
\[
\Prob(|\Delta_u|>\sigma_{a,z}t)\leq \Prob(|\text{I}|> \tfrac 12 \sigma_{a,z} t)+\Prob(|\text{II}|> \tfrac 12 \sigma_{a,z} t)\leq 4\exp(-c K\min \{t, t^2\}).
\]
Finally, recall the bound above holds for all $u=\bfe_i\pm\bfe_j$, so by \eqref{tmp:usigma}
\[
\Prob(|[\Delta]_{i,j}|>\sigma_{a,z}t)\leq \Prob(|\Delta_{\bfe_i+\bfe_j}|>\sigma_{a,z}t)+\Prob(|\Delta_{\bfe_i-\bfe_j}|>\sigma_{a,z}t)\leq 8\exp(-c K\min \{t, t^2\}). 
\]
\end{proof}
Entry-wise concentration now comes as a direct corollary. 
\begin{proof}[Proof of Theorem \ref{thm:localconcen} b)]
Let $\Delta=Z-\E Z$. Note that $\|Z-\E Z\|_\infty=\max_{i,j=1,\ldots,d} \{|[\Delta]_{i,j}|\},$ so using the previous lemma we have our claim by the union bound 
\[
\Prob(\|Z-\E Z\|_\infty>\sigma_{a,z}t)\leq \sum_{i,j} \Prob(|[\Delta]_{i,j}|>\sigma_{a,z}t)\leq  8d^2 \exp(-cK\min \{t, t^2\}). 
\]
\end{proof}
\subsection{Summation of entry-wise deviation}
One simple fact of matrix norm is that $\|\Delta\|\leq \|\Delta\|_1$. This is also exploited by \cite{BL08}
\begin{lem}
\label{lem:l2byl1}
Given a matrix $\Delta$, the following holds
\begin{enumerate}[a)]
\item If $\Delta$ is symmetric, then 
\[
\|\Delta\|\leq\|\Delta\|_1= \max_i \left\{\sum_{j=1}^d  |[\Delta]_{i,j}|\right\}.
\]
\item $\|\Delta\|_\infty\leq \|\Delta\|$ always holds.  If in addition $\Delta$ has bandwidth $l$, then 
$ \|\Delta\|\leq \calB_l  \|\Delta\|_\infty.$
\end{enumerate}
\end{lem}
\begin{proof}
For a) part, recall that $\bfe_i$ is the $i$-th standard basis vector. Notice that 
\[
\pm (\bfe_i \bfe_j^T+ \bfe_j \bfe_i^T)\preceq \bfe_i \bfe_i^T+ \bfe_j \bfe_j^T. 
\]
Therefore
\begin{align*}
\Delta&=\sum_i [\Delta]_{i,i}\bfe_i\bfe_i^T+ \frac{1}{2}\sum_{i\neq j} [\Delta]_{i,j} (\bfe_i \bfe_j^T+\bfe_j \bfe_i^T)\\
&\preceq \sum_{i=1}^d [\Delta]_{i,i}\bfe_i\bfe_i^T+\frac{1}{2}\sum_{i\neq j} |[\Delta]_{i,j}| (\bfe_i \bfe_i^T+\bfe_j \bfe_j^T)= \sum_{i=1}^d \sum_{j=1}^d |[\Delta]_{i,j}|\bfe_i\bfe_i^T\preceq \|\Delta\|_1 I_d. 
\end{align*}

For the b) part, by the definition of operator norm, and $\|\bfe_i\|=\|\bfe_j\|=1$, we have
\[
|\bfe_i \Delta \bfe^T_j|\leq \|\Delta\|.
\]
Taking maximum among all $i$ and $j$, we have $\|\Delta\|_\infty\leq \|\Delta\|$. 

Next note that $\|\Delta\|=\|\Delta\Delta^T\|^{1/2}\leq \max_i \sum_j |[\Delta\Delta^T]_{i,j}|$,  and if $\Delta$ is of bandwidth $l$, by part a)
\[
\sum_j |[\Delta\Delta^T]_{i,j}|\leq \sum_{k:\dist(i,k)\leq l}\sum_{j:\dist(j,k)\leq l} |[\Delta]_{i,k}||[\Delta]_{j,k}|\leq \calB^2_l \|\Delta\|_\infty^2.
\]
Therefore $\|\Delta\|\leq \calB_l \|\Delta\|_\infty$. 
\end{proof}
Now the Theorem \ref{thm:localconcen} a) comes as a direct corollary:
\begin{proof}[Proof of Theorem \ref{thm:localconcen} a)]
Let $\Delta=Z-\E Z$.  By Lemma \ref{lem:l2byl1} a), 
\[
\|\Delta\circ \bfD_L\|\leq \|\Delta\circ \bfD_L\|_1=\max_i\left\{\sum_{j=1}^d [\bfD_L]_{i,j} |[\Delta]_{i,j}|\right\}\leq \|\bfD_L\|_1 \max_{i,j} |[\Delta]_{i,j}|=\|\bfD_L\|_1 \|Z-\E Z\|_\infty. 
\]
Therefore by part b) of this theorem, 
\[
\Prob(\|\Delta\circ \bfD_L\| \geq \|\bfD_L\|_1\sigma_{a,z}t )\leq \Prob\left( \|Z-\E Z\|_\infty\geq  \sigma_{a,z} t\right)\leq 8\exp(2\log d- cK \min \{t, t^2\}). 
\]
\end{proof}
\section{Error analysis of LEnKF}
\label{sec:error}

\subsection{Localization inconsistency}
\begin{lem}
\label{lem:localizationerror}
Fix an $L>l$, if matrix $A$ is of bandwidth $l$, the difference caused by commuting  localization and bilinear product with $A$
\[
\Delta=A [C\circ \Dcut^L] A^T- [A CA^T]\circ \Dcut^L
\]
has nonzero entries only for indices $(i,j)$ with $|\dist(i,j)-L|\leq 2l$. 

If in addition, matrix $C$ follows an $(M,\Phi,L)$-localized structure, then
\[
|[\Delta]_{i,j}|\leq 
M \Phi(L-2l)\|A\|^2_\infty \calB_l^2, \quad   L-2l\leq \dist(i,j)\leq L.
\]
Recall that $\calB_l$ is the volume constant given by \eqref{eqn:calB}. 
\end{lem}
\begin{proof}
By the matrix product rule, 
\begin{equation}
\label{tmp:deltaloc}
[\Delta]_{i,j}=\sum_{\dist(u,v)\leq L} [A]_{i,u} [C]_{u,v} [A]_{j,v}- \unit_{\dist(i,j)\leq L} \sum_{u,v} [A]_{i,u}  [C]_{u,v} [A]_{j,v}.
\end{equation}
If $\dist(i,j)>L+2l$, note that $[A]_{i,u}[A]_{j,v}\neq 0$ only when $\dist(i,u)\leq l, \dist(j,v)\leq l$. But for these terms,  by the triangular inequality $\dist(u,v)>L$, and they are not included in \eqref{tmp:deltaloc}. Therefore \eqref{tmp:deltaloc}$=0$. 

If $\dist(i,j)\leq L$, it is easy to verify that  $[\Delta]_{i,j}=-\sum_{\dist(u,v)> L} [A]_{i,u} [C]_{u,v} [A]_{j,v}$. Moreover, $[A]_{i,u}[A]_{j,v}\neq 0$ only when $\dist(i,u)\leq l, \dist(j,v)\leq l$. So if $\dist(i,j)<L-2l$, then by triangular inequality $\dist(u,v)<L$ and $[\Delta]_{i,j}=0$. 

Next, we assume $C$ follows an $(M,\Phi, L)$-localized structure. If $L<\dist(i,j)$, then  among the nonzero terms in $[\Delta]_{i,j}=\sum_{\dist(u,v)\leq  L} [A]_{i,u} [C]_{u,v} [A]_{j,v}$, $\dist(u,v)\geq L-2l$ by triangular inequality. This leads to  
\[
|[\Delta]_{i,j}|\leq \sum_{u,v:\dist(u,v)\leq L,\\ \dist(i,u)\leq l, \dist(j,u)\leq l} \|A\|_\infty^2 M\Phi(L-2l)\leq \calB_l^2  \|A\|_\infty^2 M \Phi(L-2l).
\]
Here we used that  
\[
\#\{ (u,v): \dist(u,v)\leq L, \dist(v, i), \dist(u,i)\leq l\}\leq \#\{ (u,v): \dist(v, i), \dist(u,i)\leq l\}=\calB^2_l.
\]
If $L-2l\leq \dist(i,j)\leq L$, then by $[\Delta]_{i,j}=-\sum_{\dist(u,v)> L} [A]_{i,u} [C]_{u,v} [A]_{j,v}$, 
\[
|[\Delta]_{i,j}|\leq M \Phi(L) \sum_{\dist(u,v)>L} |[A]_{i,u}||[A]_{j,v}|\leq M \Phi(L)\|A\|_\infty^2 \calB_l^2,
\] 
where we applied the inequality 
\[
\#\{ (u,v):  \dist(v, i), \dist(u,i)\leq l, \dist(u,v)>L\}\leq \#\{ (u,v): \dist(v, i), \dist(u,i)\leq l\}=\calB^2_l.
\]
In either case, we have the bound we claim, since $\Phi(L)\leq \Phi(L-2l)$. 
\end{proof}

\begin{proof}[Proof of Proposition \ref{lem:Kalmanlocalizationerror}]
Since Schur product is a linear operation, we can decompose the localization inconsistency as
\begin{align*}
\Delta_{loc}=&[A_n (I-\Khat_n H )][\Chat_n \circ \Dcut^L][(I-\Khat_n H )^TA^T_n]\\
&-\big[[A_n (I-\Khat_n H )]\Chat_n[(I-\Khat_n H )^TA^T_n]\big]\circ\Dcut^L\\
&+[\sigma_o^2A_n\Khat_n\Khat_n^TA_n^T +\Sigma_n]-[\sigma_o^2A_n\Khat_n\Khat_n^TA_n^T +\Sigma_n]\circ\Dcut^L
\end{align*}
Since both $\Khat_n$ and $\Sigma_n$ are of bandwidth at most $l$,  $A_n\Khat_n\Khat_n^TA_n^T$ has bandwidth at most $4l$ by triangular inequality. Since $L\geq 4l$, so 
\[
[\sigma_o^2 A_n\Khat_n\Khat_n^T A_n^T+\Sigma_n]=[\sigma_o^2A_n\Khat_n\Khat_n^T A_n^T+\Sigma_n]\circ\Dcut^L,
\]
In other words, $\Delta_{loc}$ is 
\[
[A_n (I-\Khat_n H )][\Chat_n \circ \Dcut^L][(I-\Khat_n H )^TA^T_n]
-\big[[A_n (I-\Khat_n H )]\Chat_n[(I-\Khat_n H )^TA^T_n]\big]\circ\Dcut^L,
\]
which can be applied by Lemma \ref{lem:localizationerror}. Next, we try to bound $\|A_n (I-\Khat_n H) \|_\infty$. Recall that $\|H\|=1$, $\|A_n\|\leq M_A$ and Lemma \ref{lem:l2byl1} b), 
\[
\|A_n (I-\Khat_n H) \|_\infty\leq \|A_n (I-\Khat_n H) \|\leq M_A \|I-\Khat_n H\|\leq M_A (1+\|\Khat_nH\|)
\]
In domain localization \eqref{eqn:Khat1}, $\Khat_n H$ has bandwidth $l$. To see this,  note  that
\begin{align}
\notag
[\Khat_n H]_{i,j}= [\Khat_n^iH]_{i,j}&=[\Chat^i_n H^T (\sigma_o^2 I_q+ H\Chat^i_n H^T)^{-1}H]_{i,j}\\
\label{tmp:KHentry}
&=\sum_{m,k}[\Chat^i_n]_{i,o_k} [(\sigma_o^2 I_q+ H\Chat^i_n H^T)^{-1}]_{k, m}\unit_{j=o_m}. 
\end{align}
Since $\Chat^i_n$ has nonzero entries only in $\calI_i\times \calI_i$, 
\[
[(\sigma_o^2 I_q+ H\Chat^i_n H^T)^{-1}]_{k, m}=\sigma_o^{-2}\unit_{k=m}\quad \text{if}\quad \dist(o_k,i)>l \text{  or  }\dist(o_m,i)>l. 
\]
Also $[\Chat^i_n]_{i,o_k}=0$ if $\dist(o_k, i)>l$. Therefore, $[\Khat_n H]_{i,j}=0$ if $\dist(i,j)>l$.

By Lemma \ref{lem:l2byl1} b), $\|\Khat_n H\|\leq \calB_l \|\Khat_n H\|_\infty $. Since the $i$-th  row of $\Khat_n H$ is the $i$-th row of $K^i_{n}H$, so by Lemma \ref{lem:l2byl1} b), 
\[
\|\Khat_n H\|_\infty\leq \max_i \{\|K^i_{n}H\|_\infty\}\leq \max_i \{\|K^i_{n}H\|\}.
\]
Moreover, by definition \eqref{eqn:Ki} and Lemma \ref{lem:l2byl1} a)
\[
\|K^i_{n}\|\leq \|\Chat^i_n\|\|(\sigma_o^2 I+H \Chat^i_n H^T)^{-1}\|\leq \sigma_o^{-2} \|\Chat^i_n\|\leq \sigma_o^{-2} \|\Chat^i_n\|_1. 
\]
Note that $\Chat^i_n$ has nonzero entries only in  $\calI_i\times \calI_i$, by Lemma \ref{lem:l2byl1}, 
\[
\|\Chat^i_n\|_1\leq \calB_l \|\Chat^i_n\|_\infty\leq \calB_l \|\Chat_n\|_\infty.
\] 
Moreover, since  $\Chat_n$ follows an $(M_n,\Phi, L)$ structure, $\|\Chat_n\|_\infty\leq M_n$. Summing up, the  domain localized Kalman gain can be bounded by
\[
\|A_n (I-\Khat_n H) \|_\infty\leq M_A (1+ \sigma_o^{-2} \calB_l M_n).
\]
Then by Lemma \ref{lem:localizationerror}, the localization inconsistency matrix is bounded entry-wise
\[
|[\Delta]_{i,j}|\leq M_n M_A^2 (1+ \sigma_o^{-2} \calB_l M_n)^2  \calB_l^2 \Phi(L-2l),
\]
while $|[\Delta]_{i,j}|=0$ if $|\dist(i,j)-L|>2l$. So there are at most $\calB_{L,l}=\max_i\#\{j, |\dist(i,j)-L|\leq 2l\}$ nonzero entries in each row.

As a consequence
\[
\|\Delta_{loc}\|\leq \|\Delta_{loc}\|_1\leq  M_n M_A^2 (1+ \sigma_o^{-2} \calB_l M_n)^2 \calB_l^2\calB_{L,l} \Phi(L-2l).
\]
\end{proof}

\subsection{Component information gain through filtering}
One of the fundamental properties in Kalman filter is that the assimilation of observation improves estimation. Mathematically, this can be represented by that the forecast covariance matrix dominates the posterior covariance matrix. Unfortunately, with LEnKF, this natural property, $\Chat_n\succeq (I-\Khat_n  H)\Chat_n (I-\Khat_n  H)^T+\sigma_o^2\Khat_n\Khat_n^T$,  may no longer hold. However, we can still show the dominance at the diagonal entries.
\begin{prop}
\label{prop:ivar}
The assimilation step lowers the variance at each component:
\[
[\Chat_n]_{i,i}\geq [(I-\Khat_n  H)\Chat_n (I-\Khat_n  H)^T+\sigma_o^2 \Khat_n\Khat_n^T]_{i,i},\quad i=1,\cdots,d.
\]
\end{prop}
\begin{proof}
Recall that the $i$-th coordinate of $\Delta \Xhat^{(k)}_n$ is updated through the Kalman gain matrix $\Khat^i_n$. Therefore,
\[
[(I-\Khat_n  H)\Chat_n (I-\Khat_n  H)^T+\sigma_o^2\Khat_n\Khat_n^T]_{i,i}=[(I-\Khat^i_n  H)\Chat_n (I-\Khat^i_n  H)^T+\sigma_o^2\Khat^i_n(\Khat^i_n)^T]_{i,i}
\]
Moreover, in \eqref{tmp:KHentry} we have shown that $[\Khat^i_n H]_{i,j}\neq 0$ only when $\dist(i,j)\leq l$, so 
\[
[(I-\Khat^i_n  H)\Chat_n (I-\Khat^i_n  H)^T+\sigma_o^2\Khat^i_n(\Khat^i_n)^T]_{i,i}=[(I-\Khat^i_n  H)\Chat^i_n (I-\Khat^i_n  H)^T+\sigma_o^2\Khat^i_n(\Khat^i_n)^T]_{i,i}. 
\]
Note that the right side is the posterior Kalman covariance with the forecast covariance being $\Chat^i_n$. Therefore by 
\[
(I-\Khat^i_n  H)\Chat^i_n (I-\Khat^i_n  H)^T+\sigma_o^2 \Khat^i_n(\Khat^i_n)^T=
\Chat_n^i-\Chat_n^i H^T(\sigma_o^2 I_q+H \Chat_n^i H^T)^{-1} H\Chat_n^i
\preceq \Chat^i_n,
\]
we have
\[
[(I-\Khat_n  H)\Chat_n (I-\Khat_n  H)^T+\sigma_o^2\Khat_n\Khat_n^T]_{i,i}\leq [\Chat^i_n]_{i,i}=[\Chat_n]_{i,i}.
\]
\end{proof}

\subsection{Sampling error}
First, we have the following general integral lemma
\begin{lem}
\label{lem:intbypart}
If $Y$ is a nonnegative random variable that satisfies 
\[
\Prob(Y> M t)\leq 8d^2 \exp (-cK\min\{t, t^2)),\quad c>0, M\geq 1. 
\]
Then for any $\delta \in (0,1)$, if $K\geq \Gamma(M \delta^{-1}, d)$, where
\[
 \Gamma(x,d)=\max\{9x^2, \tfrac{24 }{ c} x, \tfrac{18 }{c} x^2 \log d\}.
\]
We have $\E Y\leq \delta$ and $\E Y^2\leq 2 M\delta$. 
\end{lem}
\begin{proof}
Let $\epsilon=\frac{\delta}{3M}$, and $X=Y/M$, we have $K\geq \max\{\epsilon^{-2}, \tfrac{8}{ c \epsilon}, \frac{2}{c\epsilon^2} \log d\}$, and 
\[
\Prob(X> t)\leq 8d^2 \exp (-cK\min\{t, t^2)),
\] 
We  will show that $\E X\leq 3\epsilon$ and $\E X^2\leq 6\epsilon$, which are equivalent to our claims. Recall the integration by part formula for nonnegative random variables, $\E X=\int^\infty_0 \Prob(X>x)dx$,
\begin{align*}
\E X&=\int^\epsilon_0 \Prob(X>x)dx+\int^\infty_\epsilon \Prob(X>x)dx\\
&\leq \epsilon +\int_{\epsilon}^\infty \Prob (X>t)dt\\
&\leq \epsilon+8\int_{1}^\infty  d^2\exp(-cKt)dt+8\int_{\epsilon}^1  d^2\exp(-cKt^2)dt\\
&\leq \epsilon+8\int_{\epsilon}^\infty  d^2\exp(-cKt)dt+8\int_{\epsilon}^\infty  d^2\exp(-cKt^2)dt.
\end{align*}
Note that with our requirement on $K$, $d^2\exp(-cK\epsilon )\leq 1$,
\[
\int_{\epsilon}^\infty  8d^2\exp(-cKt)dt=\frac{8d^2}{cK}\exp(-cK\epsilon)\leq \frac{8 }{ cK}\leq \epsilon.
\]
And for $t>\epsilon$, $8\leq 2\epsilon c K t$, so 
\[
\int_{\epsilon}^\infty 8\exp(-cKt^2)dt\leq 
\epsilon \int_{\epsilon}^\infty   2cKt\exp(-cKt^2)dt=\epsilon\exp(-cK\epsilon^2)\leq \epsilon. 
\]
As for $\E X^2$, we again apply the integration by part formula
\begin{align*}
\E X^2&=\int^\infty_0 2t\Prob(X\geq t)dt\\
&\leq 2\epsilon+\int^\infty_\epsilon 2t \Prob(X\geq t)dt\\ 
&\leq 2\epsilon+8d^2 \int^\infty_\epsilon 2t\exp(-cK t)dt+8d^2 \int^\infty_\epsilon 2t\exp(-cK t^2)dt\\
&=2\epsilon+16d^2 \exp(-cK\epsilon)(\tfrac{\epsilon}{cK}+\tfrac{1}{c^2K^2})+\tfrac{8d^2}{cK} \exp(-cK \epsilon^2)\\
&\leq 2\epsilon+\tfrac{16 \epsilon}{cK}+\tfrac{16}{c^2K^2}+\tfrac{8}{cK}\leq 6\epsilon.
\end{align*}
We used $K\geq \max\{\epsilon^{-2}, \tfrac{8}{ c \epsilon}, \frac{2}{c\epsilon^2} \log d\}$ in the last line. 
\end{proof}

\begin{cor}
\label{cor:concentration}
Under condition 1) of Theorem \ref{thm:main}, suppose  $\Chat_n$ follows  $(M_n,\Phi,L)$-localized structure. For any $\epsilon\in (0,1)$, if
\begin{enumerate}[a)]
\item $K>\Gamma(\calB_L\epsilon^{-1}, d)$, then the sampling error
\[
\E_n \|(\Chat_{n+1}-r\Riccati_n(\Chat_n))\circ \Dcut^L\|\leq \epsilon (\calB^2_l M_A^2M_n+M_\Sigma),
\] 
\item $K>\Gamma(rC\epsilon^{-1}, d)$ for any $C\geq 1$, then the entry-wise sampling error
\[
\E_n \|\Chat_{n+1}-r\Riccati_n(\Chat_n)\|_\infty \leq \epsilon C^{-1} \|\Riccati_n(\Chat_n)\|_\infty. 
\]
\[
\E_n \|\Chat_{n+1}-r\Riccati_n(\Chat_n)\|^2_\infty \leq \epsilon 2C^{-1} \|\Riccati_n(\Chat_n)\|^2_\infty. 
\]  
\end{enumerate}
\end{cor}
\begin{proof}
We apply Theorem \ref{thm:localconcen} with
\[
a_k=\sqrt{r} A_n (I-\Khat_nH) \Delta \Xhat^{(k)}_n,\quad z_k=\sqrt{r} A_n\Khat_n \zeta^{(k)}_n+ \sqrt{r} \xi^{(k)}_{n},
\]
and $\bfD_L=\Dcut^L$. Then 
\[
\Sigma_a=r A_n (I-\Khat_nH) \Chat_n (I-\Khat_nH)^T A_n^T,\quad
\Sigma_z=r \sigma_o^2A_n\Khat_n \Khat_n^T A_n^T+r\Sigma_n.
\]
Note that $\Sigma_a\preceq \Sigma_a+\Sigma_z$ and $\Sigma_z\preceq \Sigma_a+\Sigma_z=r\Riccati_n(\Chat_n)$, where recall
\[
\Riccati_n(\Chat_n)=A_n Q_n A_n^T+\Sigma_n,\quad Q_n:=(I-\Khat_n  H)\Chat_n (I-\Khat_n  H)^T+\sigma_o^2\Khat_n\Khat_n^T.
\]
Therefore
\[
\sigma_{a,z}\leq r\max_{i,j} \{[\Sigma_a]_{i,i}, [\Sigma_a]^{1/2}_{i,i}[\Sigma_z]^{1/2}_{j,j}\}
\leq r\max_{i} [\Sigma_a+\Sigma_z]_{i,i}= r\|\Riccati_n(\Chat_n)\|_\infty.
\]
Moreover, since $Q_n$ is positive semidefinite (PSD), so 
\[
\|Q_n\|_\infty= \max_{i,j}|[Q_n]_{i,j}|\leq \sqrt{\max_{i}[Q_n]_{i,i} \max_{j}[Q_n]_{j,j}}=\max_{i}[Q_n]_{i,i}\leq \|Q_n\|_\infty.
\] 
Moreover, by Proposition \ref{prop:ivar},
\[
[Q_n]_{i,i}\leq [\Chat_n]_{i,i}\leq M_n. 
\]
Since $\Riccati_n(\Chat_n)$ is PSD, and by Lemma \ref{lem:l2byl1} $\|A_n\|_\infty\leq \|A_n\|\leq M_A$, 
\begin{align*}
\|\Riccati_n(\Chat_n)\|_\infty \leq \max_{i}[\Riccati_n(\Chat_n)]_{i,i}&= \max_{i}\left\{[\Sigma_n]_{i,i}+\sum_j [A_n]_{i,j}[Q_n]_{j,k} [A_n]_{i,k}\right\}\\
&\leq M_A^2 \calB^2_l  M_n+M_\Sigma. 
\end{align*}
Apply Theorem \ref{thm:localconcen}, since $\|\Dcut^L\|_1=\max_i \sum_{j: \dist(i,j)<L} 1=\calB_L$,  we have that
\[
\Prob_n(\|(\Chat_{n+1}-r\Riccati_n(\Chat_n))\circ \Dcut^L\|/ \|\Riccati_n(\Chat_n)\|_\infty>r \calB_L t)
\leq 8d^2\exp(-cK\min\{t, t^2\}). 
\]
\[
\Prob_n(\|\Chat_{n+1}-r\Riccati_n(\Chat_n)\|_\infty/ \|\Riccati_n(\Chat_n)\|_\infty> r t)
\leq 8d^2\exp(-c K\min\{t, t^2\}). 
\]
$\Prob_n$ denotes the probability conditioned on $\mathcal{F}_n$. Apply Lemma \ref{lem:intbypart} with the both of them, but using  $\delta=\epsilon$ for the first inequality and $\delta=\epsilon C^{-1}$ for the second, we have our claimed results.
\end{proof}

\subsection{Error analysis}
Next, we proceed to prove Theorem \ref{thm:main}.
\begin{proof}[Proof of Theorem \ref{thm:main}]
For each time $n$, let $r_n$ be the smallest number such that the following hold,
\[
\E_{S} \ehat_n \otimes\ehat_n\preceq r_n(\Chat_n\circ\Dcut^L+ \rho I_d),\quad r_n\geq 1. 
\]
We will try to find a recursive upper bound of $r_{n+1}$ in term of  $r_n$.

\noindent\textbf{Step 1: tracking the filter error.} 
Recall that the forecast error at time $n+1$ is provided by the \eqref{tmp:errorupdate}, and its covariance conditioned on sample noise realization is 
\begin{align*}
\E_{S} \ehat_{n+1}\otimes\ehat_{n+1}&=A_n [(I-\Khat_{n} H)\E_S \ehat_n\otimes \ehat_n (I-\Khat_{n} H)^T+\sigma_o^2\Khat_{n}\Khat_{n}^T]A_n^T+\Sigma_n\\
&\preceq r_n A_n (I-\Khat_{n} H) (\Chat_n\circ \Dcut^L) (I-\Khat_{n} H)^TA_n^T\\
&\quad+ [\sigma_o^2 A_n \Khat_{n}\Khat_{n}^TA_n^T+r_n\rho A_n(I-\Khat_{n} H)(I-\Khat_{n} H)^T A_n^T+\Sigma_n].
\end{align*}
By Young's inequality $(a+b) (a+b)^T\preceq 2a a^T+2 b b^T$, and that $HH^T=I_q$,
\begin{align*}
A_n(I-\Khat_{n} H)(I-\Khat_{n} H)^T A_n^T
&\leq 2 (A_n A_n^T+A_n\Khat_n H H^T\Khat^T_nA_n^T) \\
&\leq 2 (A_n A_n^T+A_n\Khat_n \Khat^T_nA_n^T) .
\end{align*}
Moreover,  $A_n A_n^T\preceq M^2_A I_d\preceq \frac{M^2_A}{m_\Sigma}  \Sigma_n$. Denote $D_\Sigma=\max\{\frac{2M^2_A}{m_\Sigma}, \frac{2}{\sigma_o^2} \}$, then
\[
 A_n(I-\Khat_{n} H)(I-\Khat_{n} H)^T A_n^T
\preceq  D_\Sigma  (\Sigma_n+\sigma_o^2 A_n \Khat_{n} \Khat_n^T A_n^T).
\]
Furthermore, 
\begin{align*}
&\E_{S} \ehat_{n+1}\otimes \ehat_{n+1}\preceq r_n A_n (I-\Khat_{n} H) (\Chat_n\circ \Dcut^L) (I-\Khat_{n} H)^TA_n^T+ (1+r_n \rho D_\Sigma)(\sigma_o^2 A_n \Khat_{n}\Khat_{n}^TA_n^T+\Sigma_n).
\end{align*}
Recall that $\Riccati_n'(\Chat_n)$ in \eqref{tmp:realerror} is 
\[
\Riccati_n'(\Chat_n) =A_n (I-\Khat_{n} H) (\Chat_n\circ \Dcut^L) (I-\Khat_{n} H)^TA_n^T+\sigma_o^2 A_n \Khat_{n}\Khat_{n}^T A_n^T+ \Sigma_n.
\]
Therefore
\[
\E_{S} \ehat_{n+1}\otimes \ehat_{n+1}\preceq \max\{1,r_n/r, (1+r_n \rho D_\Sigma)/r\}\cdot r\Riccati'_n(\Chat_n).
\]
With our condition 2) on $\rho$, 
\[
(1+r_n \rho D_\Sigma)/r
\leq \tfrac{1}{r}+ \tfrac{r-1}{r} \tfrac{r_n}{r}\leq \max\{1, r_n/r\},
\]
so $\E_{S} \ehat_{n+1}\otimes \ehat_{n+1}\preceq \max\{1, r_n/r\} r\Riccati_n'(\Chat_n)$.\\
\newline
\noindent\textbf{Step 2: difference between filter error covariance and its estimate.} \\
The EnKF estimates the error covariance by the ensemble covariance $\Chat_{n+1}$. Its conditional expectation is 
\begin{equation}
\label{tmp:decomp}
\E_{n} \Chat_{n+1}=r\Riccati_n(\Chat_n)=r(A_n (I-\Khat_{n} H) \Chat_n (I-\Khat_{n} H)^TA_n^T+ \sigma_o^2A_n\Khat_{n}\Khat_{n}^TA_n^T+ \Sigma_n).
\end{equation}
In order to establish a control of the new filter error using localized ensemble covariance matrix, consider the difference 
\[
\Chat_{n+1}\circ \Dcut^L- r\Riccati_n'(\Chat_n)=(\Chat_{n+1}\circ \Dcut^L-\E_{n}\Chat_{n+1}\circ \Dcut^L)+r(\Riccati_n(\Chat_n)\circ\Dcut^L- \Riccati_n'(\Chat_n)). 
\]
The first part of \eqref{tmp:decomp} is the error caused by  sampling. By Corollary \ref{cor:concentration}, if we denote
\[
\mu_{n+1}:=\|\Chat_{n+1}\circ \Dcut^L-\E_{n}\Chat_{n+1}\circ\Dcut^L\|
\]
then $\E_{n} \mu_{n+1}\leq (\calB^2_l M_A^2M_n+M_\Sigma)  \delta/r $ if $K$ satisfies condition 5). 

The second part of \eqref{tmp:decomp} is the localization inconsistency.  By Proposition \ref{lem:Kalmanlocalizationerror}, we have 
\[
\|\Riccati_n (\Chat_n)\circ\Dcut^L- \Riccati_n'(\Chat_n)\| \leq   M_n M_A^2 (1+ \sigma_o^{-2} \calB_l M_n)^2 \calB_l^2\calB_{L,l} \Phi(L-2l)=:\nu_{n+1}. 
\]
Summing these two parts up, 
\[
r\Riccati_n'(\Chat_n)\preceq \Chat_{n+1}\circ \Dcut^L +r (\mu_{n+1}+\nu_{n+1})I_d. 
\]
Then
\[
r\Riccati_n'(\Chat_n)\preceq (1+ \tfrac{r}{\rho}(\mu_{n+1}+\nu_{n+1}) )(\Chat_{n+1}\circ \Dcut^L+\rho I_d). 
\]
Recall that in step 1, we have $\E_{S} \ehat_{n+1}\ehat_{n+1}^T\preceq \max\{1,\tfrac{r_{n}}{r}\}r\Riccati_n'(\Chat_n)$, so if we let $r_{n+1}$ be the smallest number such that 
\[
\E_S \ehat_{n+1}\otimes \ehat_{n+1}\preceq r_{n+1}(\Chat_{n+1}\circ \Dcut^L+\rho I_d),\quad r_{n+1}\geq 1,
\]
then 
\begin{equation}
\label{tmp:rn+1}
r_{n+1}\leq \max\{1,\tfrac{r_{n}}{r}\}(1+ \tfrac{r}{\rho}(\mu_{n+1}+\nu_{n+1})).
\end{equation}

\noindent\textbf{Step 3: long time stability analysis. } Since $r_*\leq r$ 
\[
\max\{0, \log (r_n/r)\}\leq \max\{0, \log (r_n/r_*)\}\leq \log r_n-\log r_* \unit_{r_n\geq r_*}. 
\]
Taking the logarithm of \eqref{tmp:rn+1}, and using that $\log(1+x+y^3)\leq x+2y$ for all $x,y\geq 0$, 
\begin{align*}
\log r_{n+1}&\leq \log r_n-\log r_*  \unit_{r_n\geq r_*}+ \log (1+ \tfrac{r}{\rho}(\mu_{n+1}+\nu_{n+1}))\\
&\leq \log r_n-\log r_*\unit_{r_n\geq r_*}+\tfrac{r}{\rho}\mu_{n+1}+ 2(\tfrac{r}{\rho}\nu_{n+1})^{1/3}.
\end{align*}
Sum this inequality from $n=0,\ldots, {T-1}$, we have
\[
\log r_*\sum_{n=0}^{T-1} \unit_{r_n\geq r_*} \leq \log r_0-\log r_{T}+\sum_{n=0}^{T-1}  (\tfrac{r}{\rho}\mu_{n+1}+ 2(\tfrac{r}{\rho}\nu_{n+1})^{1/3}).
\]
Because $r_{T}\geq 1$, 
\[
\sum_{n=0}^{T-1} \unit_{r_n\geq r_*}\leq \frac{\log r_0}{\log r_*}+\frac{1}{\log r_*}\sum_{n=0}^{T-1} (\tfrac{r}{\rho}\mu_{n+1}+ 2(\tfrac{r}{\rho}\nu_{n+1})^{1/3}). 
\]
Take expectation,
\begin{equation}
\label{tmp:sumprob}
\sum_{n=0}^{T-1} \Prob(r_n\geq r_*)=\E \sum_{n=0}^{T-1} \unit_{r_n\geq r_*}\leq \frac{\log r_0}{\log r_*}+\frac{1}{\log r_*} \sum_{n=0}^{T-1}   ( \tfrac{r}{\rho}\E \mu_{n+1}+ 2\E (\tfrac{r}{\rho}\nu_{n+1})^{1/3}).
\end{equation}
\noindent\textbf{Step 4: Upper bounds for \eqref{tmp:sumprob}.}  Recall  in step 2 we have that
\[
\sum_{n=0}^{T-1}    \tfrac{r}{\rho}\E \mu_{n+1}
\leq \sum_{n=0}^{T-1} \tfrac{\delta}{\rho} (\calB^2_l M_A^2\E M_n+M_\Sigma).
\]
Next, note the following holds because $\calB_l\geq 1$
\[
\nu_{n+1}=M_A^2\calB_l^2\calB_{L,l} \Phi(L-2l)  M_n  (1+ \sigma_o^{-2} \calB_l M_n)^2 \leq M^2_A \sigma_o^2\calB_l^3 \calB_{L,l} \Phi(L-2l) (1+ \sigma_o^{-2} \calB_l M_n)^3. 
\]
With condition 4),  we have
\[
M^{2/3}_A  \calB_{L,l}^{1/3}\calB^2_l\Phi^{1/3}(L-2l)\leq \delta,
\]
so
\[
\E \nu_{n+1}^{1/3}\leq \E M^{2/3}_A \sigma_o^{2/3} \calB_{L,l}^{1/3}\calB_l\Phi^{1/3}(L-2l)(1+ \sigma_o^{-2} \calB_l M_n)\leq \delta (\sigma_o^{2/3}+\sigma_o^{-1/3}\E M_{n}).
\]
In conclusion, 
\begin{align*}
2\E (\tfrac{r}{\rho}\nu_{n+1})^{1/3}  &\leq 2\delta \tfrac{r^{1/3}}{\rho^{1/3}}(\sigma_o^{2/3}+\sigma_o^{-1/3}\E M_{n}).
\end{align*}
Plug  these bounds to \eqref{tmp:sumprob}, and then use \eqref{eqn:stabloc}
\begin{align*}
\frac{1}{T}\sum_{n=0}^{T-1} \Prob(r_n\geq r_*)&
\leq \frac{r_0}{T \log r_*}+\frac{\delta}{T \log r_*} (\rho^{-1}\calB_l^2 M_A^2+\tfrac{2r^{1/3}}{(\rho\sigma_o)^{1/3}})\sum_{n=0}^{T-1}\E M_n+\frac{\delta}{ \log r_*}(\rho^{-1} M_\Sigma+2\tfrac{r^{1/3}}{\rho^{1/3}}\sigma^{2/3}_o)\\
&\leq \frac{r_0}{T \log r_*}+\frac{\delta (B_0\|C_0\|+D_0)}{T \log r_*}(\rho^{-1} \calB_l^2 M_A^2+\tfrac{2r^{1/3}}{(\rho\sigma_o)^{1/3}})\\
&\quad +\frac{\delta}{\log r_*}\left((\rho^{-1} \calB_l^2 M_A^2+\tfrac{2r^{1/3}}{(\rho\sigma_o)^{1/3}})M_0+\rho^{-1} M_\Sigma+2\tfrac{r^{1/3}}{\rho^{1/3}}\sigma^{2/3}_o  \right).
\end{align*}
For our result, simply notice that 
\[
r_n\leq r_*\quad\Leftrightarrow\quad \E_{S} \ehat_n\otimes\ehat_n\preceq  r_*(\Chat_n+\rho I_d). 
\]
\end{proof}

\section{Localized covariance for linear LEnKF systems}
\label{sec:closeneighbor}
As discussed in the introduction, the existence of a localized covariance  structure is often assumed in practice to motivate the localization technique. Our result, Theorem \ref{thm:main}, shows that such a structure indeed can guarantee estimated performance, assuming the parameters and sample size are properly tuned.  Then it is  natural to ask when does a stable localized structure exist. This is an interesting and important question by itself, but to answer it for general signal-observation systems with rigorous proof is beyond the scope of this paper.  Here we demonstrate how to verify a stable localized covariance for simple linear models.

\subsection{Localized covariance propagation with weak local interactions}
As discussed in Theorem \ref{thm:main}, we require $A_n$ to be of a short bandwidth $l$. In other words, interaction in one time step exists only for components of distance  $l$ apart. When $l=1$, this type of interaction is often called nearest neighbor interaction, and it includes many statistical physics models with proper spatial discretization.

Generally speaking,  localized covariance is  formed  through  weak local interactions. With  linear dynamics described by $A_n$, one way to enforce a weak local interaction is through \eqref{aspt:weakinter}. We will show in this subsection that weak local interaction propagates a localized covariance structure of form $[\Chat_n]_{i,j}\propto \lambda_A^{\dist(i,j)}$, from diagonal entries of the covariance matrix to entries further away from diagonal. 

To describe the state of localization in  covariance matrices $\Chat_n$ and $C_n$, we define the following quantities 
\begin{equation}
\label{eqn:loclevel}
\Mhat_{n,l}=\max_{i,j}\left\{ |[\Chat_n]_{i,j}| \lambda_A^{-\dist(i,j)\wedge l}\right\},\quad M_{n,l}=\max_{i,j}\left\{ |[C_n]_{i,j}| \lambda_A^{-\dist(i,j)\wedge l}\right\}.
\end{equation}
Then clearly, the forecast covariance matrices follow the $(M_n, \lambda^{x}_A, L)$ localized structure with $M_n=\Mhat_{n,L}$. The goal of this section is to show that $\Mhat_{n,L}$ is a stable stochastic sequence. 

The following  properties hold immediately because the matrices involved are PSD.
\begin{lem}
\label{lem:basicloc}
Given positive semidefinite (PSD) matrices $C_n, \Chat_n$, define $M_{n,l}, \Mhat_{n,l}$ as in \eqref{eqn:loclevel}, we have $\Mhat_{n,0}=\max_{i} [\Chat_n]_{i,i}$, 
\[
\Mhat_{n,0}\leq \Mhat_{n,1}\leq \cdots\leq \Mhat_{n,k} \leq \Mhat_{n,0}\lambda_A^{-k}.
\]
The same properties also hold for $M_{n,k}$ as well. 
\end{lem}
\begin{proof}
Recall that $[\Chat_n]_{i,j}$ is the ensemble covariance, so for $i\neq j$
\[
 |[\Chat_n]_{i,j}| \leq \sqrt{ |[\Chat_n]_{i,i}|  |[\Chat_n]_{j,j}|}\leq \max_{i} [\Chat_n]_{i,i}.
 \]
 Therefore 
 \[
\Mhat_{n,0}=\max_{i,j} |[\Chat_n]_{i,j}|=\max_{i} [\Chat_n]_{i,i}.
\]
The monotonicity of $\Mhat_{n,k}$ in $k$ is quite obvious since $\dist(i,j)\wedge k\leq \dist(i,j)\wedge (k+1)$, and 
\[
\Mhat_{n,k}=\max_{i,j} \left\{|[\Chat_n]_{i,j}| \lambda_A^{-\dist(i,j)\wedge k}\right\}\leq \lambda_A^{-k}\max_{i,j} |[\Chat_n]_{i,j}|.
\]
\end{proof}

Next, we investigate how does the forecast step change the state of localization.
\begin{prop}
\label{prop:localforecast}
Suppose $\Sigma_n=\sigma_\xi^2 I_d$  and the linear dynamics admits a weak local interaction satisfying \eqref{aspt:weakinter}, the forecast step propagates the localization in covariance. In particular, given any covariance matrix $C_n$, and let $\Chat_{n+1}= A_n C_n A_n^T+\Sigma_n$, then  the localization states described by \eqref{eqn:loclevel} follows
\[
\Mhat_{n+1,0}\leq \lambda^2_A M_{n,0}+\sigma_\xi^2, 
\]
\[
\Mhat_{n+1,k}\leq \max\{\lambda^2_A M_{n,k}, \Mhat_{n+1,0}\},
\]
\[
\Mhat_{n+1,k+1}\leq \max\{\lambda_A M_{n,k}, \Mhat_{n+1,0}\}.
\]
\end{prop}
\begin{proof}
Note that $[\Chat_{n+1}]_{i,j}=[A_n C_n A_n^T]_{i,j}+\sigma_\xi^2 \unit_{i=j}$. Moreover 
\begin{align*}
|[A_{n}C_n&A_n^T]_{i,j}|\leq \sum_{m,m'} |[A_n]_{i,m} [A_n]_{j,m'} [C_n]_{m,m'}|\\
&\leq \sum_{m,m'} |[A_n]_{i,m}|\lambda_A^{-\dist(i,m)} |[A_n]_{j,m'}|\lambda_A^{-\dist(j,m')} M_{n,k} \lambda_A^{\dist(i,m)+\dist(j,m')+\dist(m,m')\wedge k}\\
&\leq \sum_{m,m'} |[A_n]_{i,m}|\lambda_A^{-\dist(i,m)} |[A_n]_{j,m'}|\lambda_A^{-\dist(j,m')} M_{n,k} \lambda_A^{\dist(i,j)\wedge k}\\
&=  M_{n,k} \lambda_A^{\dist(i,j)\wedge k} \left(\sum_{m} |[A_n]_{i,m}|\lambda_A^{-\dist(i,m)}\right)\left(\sum_{m} |[A_n]_{j,m}|\lambda_A^{-\dist(j,m)}\right),
\end{align*}
which by  \eqref{aspt:weakinter} is bounded by $\lambda_A^2 M_{n,k} \lambda_A^{\dist(i,j)\wedge k}$.

By Lemma \ref{lem:basicloc},
\[
\Mhat_{n+1,0}=\max_i [\Chat_{n+1}]_{i,i}\leq \lambda_A^2 M_{n,0}+\sigma_\xi^2.
\]
Moreover, 
\[
\Mhat_{n+1,k}=\max\left\{\max_{i\neq j} [\Chat_{n+1}]_{i,j}\lambda_A^{-\dist(i,j)\wedge k}, \max_i [\Chat_{n+1}]_{i,i}\right\}\leq \max\left\{\lambda_A^2 M_{n,k},  \max_i [\Chat_{n+1}]_{i,i}\right\}.
\]
\[
\Mhat_{n+1,k+1}=\max\left\{\max_{i\neq j} [\Chat_{n+1}]_{i,j}\lambda_A^{-\dist(i,j)\wedge (k+1)}, \max_i [\Chat_{n+1}]_{i,i}\right\}\leq \max\left\{\lambda_A M_{n,k},  \max_i [\Chat_{n+1}]_{i,i}\right\}.
\]
\end{proof}

\subsection{Preserving  a localized structure with sparse observations}
From now on, we require the observations to be sparse in the sense that $\dist(o_i, o_j)>2l$ for any $i\neq j$. Then for each location $i\in \{1,\cdots, d\}$, there is at most one location $o(i)\in\{o_1,\cdots, o_q\}$ such that $\dist(i,o(i))\leq l$. If such an $o(i)$ doesn't exist, we set $o(i)=nil$, the analysis step will not update it, and we will see the discussion for these components are trivial. 

With domain localization  and sparse observations,  the analysis step updates the information at  the $i$-th component using only the observation at $o(i)$. This significantly simplifies the formulation of $(H \Chat^i_n  H^T+ \sigma_o^2 I_q)^{-1}$, which is diagonal with entries $(\sigma_o^2+[\Chat_n]_{o_i, o_i})^{-1}$ in $\calI_i\times \calI_i$. As a result, the Kalman update matrix has entries 
\[
[\Khat_{n}H]_{i,j}=[\Khat_{n}^iH]_{i,j}=\begin{cases}\frac{[\Chat_n]_{i,o(i)}}{\sigma_o^2+ [\Chat_n]_{o(i),o(i)} },\quad &j=o(i);\\  0,  &\text{else}.\end{cases}
\]
In fact, if we apply the covariance localization scheme instead of domain localization, the Kalman gain remains the same in this setting. 

In below, we investigate how does the assimilation step change the state of localization.
\begin{prop}
\label{prop:localanalysis}
Given any covariance matrix $\Chat_n$, define $\Khat_n$ as the Kalman gain in \eqref{eqn:Khat1}, and let
\[
C_{n}= (I-\Khat_nH )\Chat_n(I-\Khat_n H)^T+\sigma_o^2 \Khat_n \Khat_n^T.
\] 
Define the  state of localization using \eqref{eqn:loclevel}. Then 
\[
M_{n,0}\leq \Mhat_{n,0},\quad M_{n,k}\leq \phi(\Mhat_{n,k})
\]
where
\[
\phi (M)= M(1+\sigma_o^{-2} M)^2 +\sigma_o^{-2} M^2.
\]
\end{prop}
\begin{proof}
Based on Lemma \ref{lem:basicloc},  $M_{n,0}=\max_i |[C_n]_{i,i}|, \Mhat_{n,0}=\max_i |[\Chat_n]_{i,i}|$, so $M_{n,0}\leq \Mhat_{n,0}$ holds by Proposition \ref{prop:ivar}.
Next, we look at the off diagonal terms:
  \begin{align}
  \notag
[C_n]_{i,j}&=[\Chat_n]_{i,j}-\frac{[\Chat_n]_{i,o(i)}[\Chat_n]_{j,o(i)} }{\sigma_o^2+[\Chat_n]_{o(i),o(i)}}-\frac{[\Chat_n]_{i,o(j)}[\Chat_n]_{j,o(j)}}{\sigma_o^2+[\Chat_n]_{o(j),o(j)}}\\
\notag
&\quad +\frac{[\Chat_n]_{i,o(i)}[\Chat_n]_{j,o(j)}[\Chat_n]_{o(i),o(j)} }{(\sigma_o^2+[\Chat_n]_{o(i),o(i)})(\sigma_o^2+[\Chat_n]_{o(j),o(j)})}\\
\label{eqn:Cndecomp}
&\quad +\frac{\sigma_o^2 [\Chat_n]_{i,o(i)}[\Chat_n]_{j,o(i)} }{(\sigma_o^2+[\Chat_n]_{o(i),o(i)})^2}\unit_{o(i)=o(j)}.
\end{align}
We have the following bounds for each term in \eqref{eqn:Cndecomp}
\begin{align*}
\left|\frac{[\Chat_n]_{i,o(i)}[\Chat_n]_{j,o(i)}}{\sigma_o^2+[\Chat_n]_{o(i),o(i)}}\right|
\leq \sigma_o^{-2} \Mhat^2_{n,k} \lambda_A^{\dist(i,o(i))\wedge k+\dist(j,o(i))\wedge k}\leq \sigma_o^{-2} \Mhat^2_{n,k} \lambda_A^{\dist(j,i)\wedge k}.
\end{align*}
\begin{align*}
&\left| \frac{[\Chat_n]_{i,o(i)}[\Chat_n]_{j,o(j)}[\Chat_n]_{o(i),o(j)} }{(\sigma_o^2+[\Chat_n]_{o(i),o(i)})(\sigma_o^2+[\Chat_n]_{o(j),o(j)})}\right |\\
&\leq \sigma_o^{-4} \Mhat^3_{n,k}\lambda_A^{ \dist(i,o(i))\wedge k+ \dist(j,o(j))\wedge k +\dist(o(j),o(i))\wedge k}
\leq \sigma_o^{-4} \Mhat^3_{n,k} \lambda_A^{\dist(i,j)\wedge k}.
\end{align*}
 \begin{align*}
\left| \frac{[\Chat_n]_{i,o(i)}[\Chat_n]_{j,o(i)}}{(\sigma_o^2+[\Chat_n]_{o(i),o(i)})^2}\right|
\leq \sigma_o^{-4}   \Mhat^2_{n,k}\lambda_A^{\dist(i,o(i))\wedge k+\dist(j,o(i))\wedge k}\leq \sigma_o^{-4}  \Mhat^2_{n,k}\lambda_A^{\dist(i,j)\wedge k}.
 \end{align*}
In summary
\[
|[C_n]_{i,j}|\leq  \Mhat_{n,k} [(1+\sigma_o^{-2} \Mhat_{n,k})^2+\sigma_o^{-2}\Mhat^2_{n,k} ] \lambda_A^{\dist(i,j)\wedge k}=\phi(\Mhat_{n,k})\lambda_A^{\dist(i,j)\wedge k}.
\]
\end{proof}

\begin{prop}
\label{prop:Kalman}
Denote $\delta_{n+1}=\lambda_A^{-L} \|\Chat_{n+1}-r\Riccati_{n} (\Chat_n)\|_\infty/\|\Riccati_{n} (\Chat_n)\|_\infty,$
and
\[
\psi_{\lambda_A}(M,\delta)=(r+\delta)\max\left\{\lambda_A M\left(1+\sigma_o^{-2}M\right)^2+\lambda_A \sigma_o^{-2}M^2, \lambda^2_A M+\sigma_\xi^2\right\}.
\]
Then for $k\leq L-1$,
\[
\Mhat_{n+1,0}\leq (r+\delta_{n+1})(\lambda^2_A M_{n,0}+\sigma_\xi^2) ,\quad
\Mhat_{n+1,k+1}\leq \psi_{\lambda_A}( \Mhat_{n,k},\delta_{n+1}).
\]
\end{prop}
\begin{proof}
Recall that 
\[
\Riccati_n(\Chat_n)= A_n [(I-\Khat_nH )\Chat_n(I-\Khat_n H)^T+\sigma_o^2 \Khat_n \Khat_n^T]A_n^T+\Sigma_n.
\] 
Following \eqref{eqn:loclevel}, we define its localized status:
\[
R_{n,l}=\max_{i,j}\left\{ | [(I-\Khat_nH )\Chat_n(I-\Khat_n H)^T+\sigma_o^2 \Khat_n \Khat_n^T]_{i,j}| \lambda_A^{-\dist(i,j)\wedge l}\right\},
\]
\[
\Rhat_{n+1,l}=\max_{i,j}\left\{ |[\Riccati_n(\Chat_n)]_{i,j}| \lambda_A^{-\dist(i,j)\wedge l}\right\}.
\]
Apply Proposition  \ref{prop:localanalysis},
\[
R_{n,0}\leq \Mhat_{n,0},\quad R_{n,k}\leq \phi(\Mhat_{n,k}). 
\]
Then apply Proposition \ref{prop:localforecast}, we find that 
\[
\Rhat_{n+1,0}=\|\Riccati_n(\Chat_n)\|_\infty\leq \lambda_A^2\Mhat_{n,0}+\sigma_\xi^2,\quad \Rhat_{n+1,k+1}\leq \max\{\lambda_A\phi(\Mhat_{n,k}), \Rhat_{n+1,0}\}.
\]
Finally by Lemma \ref{lem:basicloc},
\begin{align*}
\Mhat_{n+1,0}=\|\Chat_{n+1}\|_\infty&\leq r \|\Riccati_n(\Chat_n)\|_\infty+ \|\Chat_{n+1}-r\Riccati_n(\Chat_n)\|_\infty\leq (r+\lambda_A^L\delta_{n+1})\|\Riccati_n(\Chat_n)\|_\infty.
\end{align*}
Since $\|\Riccati_n(\Chat_n)\|_\infty\leq \lambda_A^2\Mhat_{n,0}+\sigma_\xi^2$, we have our bound for $\Mhat_{n+1,0}$. Likewise,
\begin{align*}
\Mhat_{n+1,k+1}&=\max_{i,j}|[\Chat_{n+1}]_{i,j}| \lambda_A^{-\dist(i,j)\wedge (k+1)}\\
&\leq r\max_{i,j}|[\Riccati_n(\Chat_n)]_{i,j}|\lambda_A^{-\dist(i,j)\wedge (k+1)}+ \max_{i,j}|[\Chat_{n+1}]_{i,j}-r[\Riccati_n(\Chat_n)]_{i,j}| \lambda_A^{-L}\\
 &= r \Rhat_{n+1, k+1}+\delta_{n+1}\|\Riccati_n(\Chat_n)\|_\infty\\
 &\leq r\max\{\lambda_A\phi(\Mhat_{n,k}), \Rhat_{n+1,0}\}+\delta_{n+1} \|\Riccati_n(\Chat_n)\|_\infty\leq \psi_{\lambda_A}(\Mhat_{n,k},\delta_{n+1}).
\end{align*}

%
%
%
\end{proof}

\subsection{Stability of localized structures}

\begin{lem}
\label{lem:diagonal}
Under the conditions of Theorem \ref{thm:formloc}, when $K>\Gamma(r\epsilon^{-1},d)$ with $\epsilon=\min\{ \tfrac{1}{2\lambda_A}-\tfrac{r}{2},\tfrac{\delta}{2}\}$, the diagonal status defined by \eqref{eqn:loclevel} satisfies:
\[
\E_n \Mhat_{n+1,0}\leq  \lambda_A \Mhat_{n,0}+(r+\delta)\sigma_\xi^{2}\quad a.s..
\]
\[
\E_n \Mhat_{n+1,0}^2\leq  \lambda_A \Mhat^2_{n,0}+\frac{(r+\delta)^2\sigma_\xi^4}{1-\lambda_A}\quad a.s..
\]
Therefore, by Gronwall's inequality,
\[
\E_0 \Mhat_{n,0}\leq \lambda^n_A\Mhat_{0,0}+(r+\delta)\sigma_\xi^2\sum_{k=0}^n \lambda_A^k
\leq \lambda^n_A\Mhat_{0,0}+\frac{(r+\delta)\sigma_\xi^2}{1-\lambda_A}\quad a.s..
\]
\[
\E_0 \Mhat^2_{n,0}\leq \lambda^{n}_A \Mhat^2_{0,0}+\frac{(r+\delta)^2\sigma_\xi^4}{(1-\lambda_A)^2}\quad a.s..
\]
\end{lem}
\begin{proof}
We apply Lemma \ref{lem:basicloc}, Proposition \ref{prop:localanalysis} to find that 
\[
\|(I-\Khat_n H)\Chat_n(I-\Khat_n H)^T+\sigma_o^2 \Khat_n \Khat_n^T\|_\infty=M_{n,0}
\leq \Mhat_{n,0}=\|\Chat_n\|_\infty,
\]
and by the first claim of  Proposition \ref{prop:localforecast},
\[
\|\Riccati_n(\Chat_n)\|_\infty \leq \lambda_A^2 \|(I-\Khat_n H)\Chat_n(I-\Khat_n H)^T+\sigma_o^2 \Khat_n \Khat_n^T\|_\infty+\sigma_\xi^2 \leq \lambda_A^2 \|\Chat_n\|_\infty+\sigma_\xi^2.
\]
Also by Young's inequality, one can show that 
\[
\|\Riccati_n(\Chat_n)\|^2_\infty \leq (\lambda_A^2 \|\Chat_n\|_\infty+\sigma_\xi^2)^2
\leq \lambda_A^3\|\Chat_n\|_\infty^2+\frac{\sigma_\xi^4}{1-\lambda_A}.
\]
With $\epsilon=\min\{ \tfrac{1}{2\lambda_A}-\tfrac{r}{2},\tfrac{\delta}{2}\}$, when $K>\Gamma(r\epsilon^{-1},d)$, by Corollary \ref{cor:concentration} b), 
\[
\E_n \|\Chat_{n+1}-r\Riccati_n(\Chat_n)\|_\infty\leq \epsilon \|\Riccati_n(\Chat_n)\|_\infty\quad a.s.,
\]
\[
\E_n \|\Chat_{n+1}-r\Riccati_n(\Chat_n)\|^2_\infty\leq 2\epsilon r\|\Riccati_n(\Chat_n)\|^2_\infty\quad a.s..
\]
By $\epsilon+r\leq \lambda_A^{-1}$ and $\|\Riccati_n(\Chat_n)\|_\infty\leq \lambda_A^2 \|\Chat_n\|_\infty+\sigma_\xi^2$,
\begin{align*}
\E_n \|\Chat_{n+1}\|_\infty &\leq \E_n \|\Chat_{n+1}-r\Riccati_n(\Chat_n)\|_\infty+r\|\Riccati_n(\Chat_n)\|_\infty\\ 
&\leq (r+\epsilon) \|\Riccati_n(\Chat_n)\|_\infty\leq\lambda_A \|\Chat_n\|_\infty+(r+\delta)\sigma_\xi^2.
\end{align*}
Likewise, because $ (r+2\epsilon)\leq \lambda^{-1}_A$, 
\begin{align*}
\E_n \|\Chat_{n+1}\|^2_\infty &\leq \E_n \|\Chat_{n+1}-r\Riccati_n(\Chat_n)\|^2_\infty+
r^2\|\Riccati_n(\Chat_n)\|_\infty^2\\ 
&\quad+2r\|\Riccati_n(\Chat_n)\|_\infty\E_n \|\Chat_{n+1}-r\Riccati_n(\Chat_n)\|_\infty\\
&\leq (2\epsilon r+r^2+2\epsilon r)\|\Riccati_n(\Chat_n)\|^2_\infty \\
&\leq (r+2\epsilon)^2\|\Riccati_n(\Chat_n)\|^2_\infty\\
&\leq \lambda_A \|\Chat_n\|^2_\infty+\frac{(r+\delta)^2\sigma_\xi^4}{1-\lambda_A}.
\end{align*}
\end{proof}

\begin{lem}
\label{lem:multistep}
Suppose the  following holds 
\[
n_*\geq 2L+\frac{\log 4\delta_*^{-1}}{\log \lambda^{-1}_A},\quad \delta_*\leq \frac14,\,\,\delta_*\leq \frac12(\lambda_A^{-1}-r),
\]
and the sample size satisfies \eqref{tmp:sampleK}. Then
\[
\E_0 \Mhat_{n_*,L}\leq \frac{1}{2} \Mhat_{0,L}+ (1+2\delta_*)M_*\quad a.s..
\]
\end{lem}
\begin{proof}
\textbf{Case 1}: if $\Mhat_{0,L}>\frac{4(r+\delta_*)\sigma_\xi^2}{\lambda_A^{L}(1-\lambda_A)} $. By Lemma \ref{lem:basicloc}
\[
\Mhat_{k,0}\leq \Mhat_{k,L}\leq \lambda_A^{-L} \Mhat_{k,0}.
\]
Then by Lemma \ref{lem:diagonal}
\[
\E_0 \Mhat_{n,L}\leq \E_0 \lambda_A^{-L}\Mhat_{n,0}\leq  \lambda_A^{n-L}  \Mhat_{0,0} +\frac{(r+\delta_*)\sigma_\xi^2}{\lambda_A^{L}(1-\lambda_A)}\leq \lambda_A^{n-L}  \Mhat_{0,L} +\frac{(r+\delta_*)\sigma_\xi^2}{\lambda_A^{L}(1-\lambda_A)}\quad a.s..
\]
By our choice of  $n_*$ , $\lambda_A^{n_*-L}\leq \frac14$, so we have our  claim, since
\[
\E_0 \Mhat_{n_*,L}\leq \frac{1}{4} \Mhat_{0,L} +\frac{(r+\delta_*)\sigma_\xi^2}{\lambda_A^{L}(1-\lambda_A)}\leq \frac{1}{2}\Mhat_{0,L}\quad a.s..
\]
\textbf{Case 2}: if $\Mhat_{0,L}\leq \frac{4(r+\delta_*)\sigma_\xi^2}{\lambda_A^{L}(1-\lambda_A)}$. Consider the event 
\[
\mathcal{U}=\{\delta_k\leq \delta_*,\forall k\leq n_*\}.
\]
Denote its complementary set as $\mathcal{U}^c$. Then the expectation can be decomposed as
\[
\E_0 \Mhat_{n_*,L}\leq \E_0 \Mhat_{n_*,L} \unit_{\mathcal{U}} +\E_0 \Mhat_{n_*,L} \unit_{\mathcal{U}^c}\leq  \E_0 \Mhat_{n_*,L} \unit_{\mathcal{U}}+\sqrt{\Prob_0(\mathcal{U}^c)} \sqrt{\E_0 \Mhat^2_{n_*,L} },\] 
where we applied the Cauchy inequality for the $\mathcal{U}^c$ part, and $\Prob_0$ is the probability conditioned on $\mathcal{F}_0$.  We will find a bound for each of the two parts. 

If $\mathcal{U}$ holds, then $\delta_{n+1}\leq \delta_*$ for $n\leq n_*-1$. By Proposition \ref{prop:Kalman}, 
\[
\Mhat_{n+1,0}\leq (r+\delta_*) (\lambda_A^2 \Mhat_{n,0}+\sigma_{\xi}^2)\leq  \lambda_A \Mhat_{n,0}+(r+\delta_*)\sigma_\xi^2. 
\]
Then by the Gronwall's inequality, under $\mathcal{U}$, 
\[
\Mhat_{n, 0}\leq \lambda^{n}_A\Mhat_{0,0}+\frac{(r+\delta_*)\sigma_\xi^2}{1-\lambda_A}. 
\]
Because $\Mhat_{0,0}\leq \Mhat_{0,L}\leq \frac{4(r+\delta_*)\sigma_\xi^2}{\lambda_A^{L}(1-\lambda_A)}$, so after $n_0=n_*-L\geq L+\lceil-\log (4r\delta_*^{-1}+4)/\log \lambda_A\rceil$
\[
\lambda^{n_0}_A\Mhat_{0,0}\leq \frac{\delta_*\sigma_\xi^2}{1-\lambda_A},\quad\text{so}\quad\Mhat_{n_0, 0}\leq \frac{(r+2\delta_*)\sigma_\xi^2}{1-\lambda_A}\leq M_*.
\]
In the next $1\leq k\leq L$ steps, since $\delta_n\leq \delta_*$ when $\mathcal{U}$ holds,  because $\psi_{\lambda_A}$ is increasing,  by Proposition \ref{prop:Kalman} 
\[
\Mhat_{n_0+k, k}\leq \psi_{\lambda_A}(\Mhat_{n_0+k-1,k-1},\delta_* ),
\]
we can derive that  $\Mhat_{n_0+L, L}\leq M_*$. Therefore  by $n_*=n_0+L$, 
\[
\E_0 \Mhat_{n_*,L}  \unit_{\mathcal{U}}\leq  M_*,\quad a.s..
\]
In order to conclude our claim, it suffices to show that 
\begin{equation}
\label{tmp:outlier}
\Prob_0(\mathcal{U}^c) \E_0 \Mhat^2_{n,L} \leq \delta_*^2 M_*^2,\quad a.s.. 
\end{equation}
Apply Lemma \ref{lem:diagonal} with $\delta=\delta_*$, recall that $\Mhat_{0,L}\leq \frac{4(r+\delta_*)\sigma_\xi^2}{\lambda_A^{L}(1-\lambda_A)}$ and $16 \lambda_A^{n_*-2L}\leq 1$,
\begin{align*}
\E_0 \Mhat^2_{n_*,L}\leq &\lambda_A^{-2L}\E_0 \Mhat^2_{n_*,0}\leq  \lambda^{n_*-2L}_A \Mhat^2_{0,L}+\frac{(r+\delta_*)^2\sigma_\xi^4}{(1-\lambda_A)^2}\\
&\leq \lambda^{n_*}_A\frac{16(r+\delta_*)^2\sigma_\xi^4}{\lambda_A^{2L}(1-\lambda_A)^2}+\frac{(r+\delta_*)^2\sigma_\xi^4}{(1-\lambda_A)^2}\leq 2M_*^2.
\end{align*}
Moreover,  by Theorem \ref{thm:localconcen}  b)
\[
\Prob(\delta_{n+1}>\delta_*|\mathcal{F}_n)\leq 8d^2\exp(-cK\lambda_A^{2L}\delta_*^2)\leq \frac{\delta^2_*}{2n_*}.
\]
where the final bound comes with the sample $K$ satisfying \eqref{tmp:sampleK}.
Therefore, by the law of iterated expectation, 
\[
\Prob_0(\mathcal{U}^c)\leq  \sum_{k=1}^{n_*} \Prob_0(\delta_k>\delta_*)=\sum_{k=0}^{n_*-1}\E_0 \Prob(\delta_{k+1}>\delta_*|\mathcal{F}_k)\leq \frac12\delta_*^2,
\]
and \eqref{tmp:outlier} comes as a result. 
\end{proof}

\begin{proof}[Proof of Theorem \ref{thm:formloc}]
Recall that $M_n=\Mhat_{n,L}$. So 
\[
\E_0 M_{n_*}\leq \frac{1}{2} M_{0}+ (1+\delta_*)M_*
\]
has been proved by Lemma \ref{lem:multistep}. This leads to the following using Gronwall's inequality,
\[
\E_0 M_{jn_*}\leq \frac{1}{2^j} M_{0}+2(1+\delta_*)M_*.
\]
Next,  for $k=1,\cdots, n_*-1$,  apply Lemma \ref{lem:diagonal} with $\delta=\delta_*$
\[
\E M_{k}=\E \Mhat_{k,L}\leq \lambda_A^{-L} \E \Mhat_{k,0}\leq \lambda_A^{-L} \E\Mhat_{0,0}+\frac{(r+\delta_*)\sigma_\xi^2}{\lambda_A^{L}(1-\lambda_A)}\leq \lambda_A^{-L}(\E \|\Chat_0\|+M_*), 
\]
because $\Mhat_{0,0}=\|\Chat_0\|_\infty\leq \|\Chat_0\|$ by Lemma \ref{lem:basicloc}. Then if $k+mn_*\leq T$, 
\begin{align*}
\sum_{j=0}^m \E M_{k+j n_*}=\sum_{j=0}^m \E\E_{k} M_{k+j n_*}
&\leq \sum_{j=0}^m  \frac{1}{2^j} \E M_{k}+2(1+\delta_*)M_*\\
&\leq  2\E\Mhat_{k,L}+2(m+1)(1+\delta_*)M_*\\
&\leq 2\lambda_A^{-L}(\E \|\Chat_0\|+M_*)+2(m+1)(1+\delta_*)M_*. 
\end{align*}
Summation of the inequality above with $k=0,\cdots,n_*-1$, we obtain our final claim. 
\end{proof}

\subsection{Small noise scaling}
\begin{proof}[Proof of Theorem \ref{thm:accuracy}]
It suffices to verify the conditions of Theorems \ref{thm:main} and \ref{thm:formloc} under the small noise scaling. 

First, we check Theorem \ref{thm:formloc}.  Condition 1)  is invariant except that $\Sigma_n=\epsilon \sigma_\xi^2 I_d$. Condition 2) concerns only of $A_n$, so it and $\lambda_A$ are also invariant under small noise scaling. For condition 3), if it holds without small noise scaling, that is
\[
(r+\delta_*)\max\left\{\lambda_A M_*\left(1+\sigma_o^{-2}M_*\right)^2+\lambda_A \sigma_o^{-2} M_*^2, \lambda^2_A M_*+\sigma_\xi^2\right\}\leq M_*.
\]
This leads to
\[
(r+\delta_*)\max\left\{\lambda_A (\epsilon M_*)\left(1+(\epsilon\sigma^2_o)^{-1}(\epsilon M_*)\right)^2+\lambda_A (\epsilon \sigma_o^2)^{-1} (\epsilon M_*)^2, \lambda^2_A \epsilon M_*+\epsilon\sigma_\xi^2\right\}\leq \epsilon M_*.
\]
Moreover, condition 3) requires that 
\[
M_*\geq \frac{(r+\delta_*) \sigma_\xi^2}{1-\lambda_A}\quad\Rightarrow\quad \epsilon M_*\geq \frac{(r+\delta_*) \epsilon \sigma_\xi^2}{1-\lambda_A}. 
\]
Therefore, with small scaling, condition 3) holds with the same $\delta_*$, while $M_*$ is replaced by $\epsilon M_*$. Condition 4) is invariant under the small noise scaling, since $\delta_*$ and $\lambda_A$ are invariant. 

As a consequence, Theorem \ref{thm:formloc} implies the following:
\begin{equation}
\label{tmp:smallscale}
\frac{1}{T}\sum_{k=1}^T\E  M_{k}\leq \frac{2n_* }{T\lambda_A^{L}}(\E \|\Chat_0\|+\epsilon M_*)+2(1+\delta_*)\epsilon M_*.
\end{equation}
This yields the first claimed result, since $M_k=\Mhat_{k,L}\geq \|\Chat_k\|_\infty$ by Lemma \ref{lem:basicloc}. 

Next we check the conditions of Theorem \ref{thm:main}. For condition 1), $m_\Sigma$ and $M_\Sigma$ need to be replaced by $\epsilon \sigma_\xi^2$ since we assume $\Sigma_n=\epsilon \sigma_{\xi}^2 I_d$. Condition 2) still holds with $(r_0,\rho)\to (\epsilon^{-1}r_0, \epsilon \rho)$ since
\[
\E\ehat_0 \otimes \ehat_0\preceq r_0(\Chat_0+\rho I_d)\quad \Rightarrow\quad \E\ehat_0 \otimes \ehat_0\preceq (\epsilon^{-1} r_0)(\Chat_0+\epsilon\rho I_d). 
\] 
Condition 3) is guaranteed by \eqref{tmp:smallscale} above, with $M_0=2(1+\delta_*)\epsilon M_*$. Condition 4) and condition 5) are both invariant, as it concerns only geometry quantities. Finally it suffices to plug in all the estimates for the result, and find
\begin{align*}
1-\frac{1}{T}\sum_{n=0}^{T-1} &\Prob(\E_{S} \ehat_n \otimes \ehat_n\preceq r_*( \Chat_n\circ \Dcut^L+\epsilon \rho I_d))\\
&\leq \frac{r_0}{T \epsilon \log r_*}+\frac{2\delta n_* (\E \|\Chat_0\|+\epsilon M_*)}{T \epsilon \lambda_A^{L} \log r_*}(\rho^{-1} \calB_l^2 M_A^2+\tfrac{2r^{1/3}}{(\rho\sigma_o)^{1/3}})\\
&\quad +\frac{\delta}{\log r_*}\left(2(\rho^{-1} \calB_l^2 M_A^2+\tfrac{2r^{1/3}}{(\rho\sigma_o)^{1/3}})(1+\delta_*) M_*+\rho^{-1} \sigma_\xi^2+2\tfrac{r^{1/3}}{\rho^{1/3}}\sigma^{2/3}_o  \right).
\end{align*}
Note that in above some $\epsilon$ terms are upper-bounded by $1$, so the inequality has a simpler form. 
\end{proof}

\section{Conclusion and discussion }
\label{sec:conclude}
Ensemble Kalman filter (EnKF) is a popular tool for high dimensional data assimilation problems. Domain localization is an important EnKF technique that exploits the natural localized covariance structure, and simplifies the associated sampling task. We rigorously investigate the performance of localized EnKF (LEnKF) for linear systems.  We show in Theorem \ref{thm:main} that in order for the filter error covariance to be dominated by the ensemble covariance, 1) the sample size $K$ needs to exceed a constant that depends on the localization radius and the logarithmic of the state dimension, 2) the forecast covariance has a stable localized structure. Condition 2) is necessary for an intrinsic localization inconsistency to be bounded. This condition  is usually assumed in LEnKF operations, but it can also be verified for systems with weak local interaction and sparse observation by Theorem \ref{thm:formloc}. 

While the results here provide the first successive explanation of LEnKF performance with almost dimension independent sample size, there are several issues that require further study. In below we discuss a few of them.
\begin{enumerate}
\item There are several ways to apply the localization technique in EnKF. We discuss here only the domain localization with  standard EnKF procedures. In principle, our results can be generalized to the covariance localization/tempering technique, and also the popular ensemble square root implementation. But such generalization will not be trivial, as the Kalman gain will not be of a small bandwidth, and localization techniques will have unclear impact on the square root SVD operation.
\item This paper studies the sampling effect of LEnKF and shows the sampling error is controllable. Yet  LEnKF without sampling error, in other words, LEnKF in the large ensemble limit, is not well studied mathematically. The effect of the localization techniques on the classical Kalman filter controllability and observability condition is not known. This may lead to practical guidelines in the choice of localization radius.  
\item Theorem \ref{thm:formloc} provides the first proof that LEnKF covariance has a stable localized structure. But the conditions we impose here are quite strong, while localized structure is taken for granted in practice. How to show it in general nonlinear settings is a very interesting question.
\end{enumerate}

\section*{Acknowledgement}
This research is supported by the NUS grant R-146-000-226-133, where X.T.T. is the principal investigator.  The author thanks Andrew J. Majda, Lars Nerger and Ramon van Handel for their discussion on various parts of this paper. 

\bibliographystyle{unsrt}
\bibliography{EnKF}

\end{document}